\documentclass{amsart}
\usepackage[utf8]{inputenc}
\usepackage{hyperref}
\usepackage{enumerate} 
\usepackage{upref}
\usepackage{verbatim} 
\usepackage{color}
\usepackage{graphicx}
\usepackage{bbm}
\usepackage[textsize=tiny]{todonotes}

\newcommand{\sign}{\mathop{\rm sign}}
\newcommand*{\mailto}[1]{\href{mailto:#1}{\nolinkurl{#1}}}
\newcommand{\arxiv}[1]{\href{http://arxiv.org/abs/#1}{arXiv:#1}}
\usepackage{amssymb, amsmath, amsfonts, amsthm}
\usepackage[all]{xy}

\usepackage{marginnote}

\newcommand{\RN}[1]{%
  \textup{\uppercase\expandafter{\romannumeral#1}}%
}

\DeclareMathOperator{\id}{Id}
\DeclareMathOperator{\meas}{meas}
\DeclareMathOperator{\supp}{supp}

\newcommand{\D}{\ensuremath{\mathcal{D}}}
\newcommand{\G}{\ensuremath{\mathcal{G}}}
\newcommand{\F}{\ensuremath{\mathcal{F}}}

\newcommand{\Real}{\mathbb R}

\newcommand{\norm}[1]{\left\Vert#1\right\Vert}

\newtheorem{theorem}{Theorem}[section]

\newtheorem{lemma}[theorem]{Lemma}
\newtheorem{definition}[theorem]{Definition}

\newtheorem{corollary}[theorem]{Corollary}
\numberwithin{equation}{section}

\allowdisplaybreaks

\begin{document}

\title[Uniqueness for the CH equation]{Uniqueness of dissipative solutions for the Camassa--Holm equation}

\author[K. Grunert]{Katrin Grunert}
\address{Department of Mathematical Sciences\\ NTNU Norwegian University of Science and Technology\\ NO-7491 Trondheim\\ Norway}
\email{\mailto{katrin.grunert@ntnu.no}}
\urladdr{\url{https://www.ntnu.edu/employees/katrin.grunert}}

\thanks{We acknowledge support by the grant {\it Wave Phenomena and Stability --- a Shocking Combination (WaPheS)}  from the Research Council of Norway. In addition, we were supported by the Swedish Research Council under the grant no. 2021-06594 while visiting the Institut Mittag-Leffler in Djursholm, Sweden, during the fall semester 2023.}  
\subjclass{Primary: 35A02, 35L45 Secondary: 35B60}
\keywords{Camassa--Holm equation, uniqueness, dissipative solutions}

\begin{abstract} 
We show that the Cauchy problem for the Camassa--Holm equation has a unique, global, weak, and dissipative solution for any initial data $u_0\in H^1(\mathbb{R})$, such that $u_{0,x}$ is bounded from above almost everywhere. In particular, we establish a one-to-one correspondence between the properties specific to the dissipative solutions and a solution operator associating to each initial data exactly one solution.
\end{abstract}

\maketitle

\section{Introduction}\label{formal}
The Camassa--Holm (CH) equation, which reads 
\begin{equation}\label{CH}
u_t-u_{txx}+3uu_x-2u_xu_{xx}-uu_{xxx}=0,
\end{equation}
was first studied in the context of water waves in \cite{CH, CHH}. Since then numerous works have been devoted to the study of the CH equation due to its rich mathematical structure and many interesting properties. For example, it has a bi-Hamiltonian structure \cite{FF}, is formally integrable \cite{Co}, and has infinitely many conserved quantities \cite{L}. 
 
 One property of solutions to the CH equation has attracted particular attention: wave breaking. That is, even for smooth initial data, classical solutions might only exist locally, since $u_x(t,x)$ might become unbounded from below pointwise within finite time. In addition, energy concentrates on sets of measure zero. This combination yields that weak solutions in $H^1(\Real)$ exist globally, but might not be unique. If the latter is the case, there exist infinitely many ways of prolonging the weak solution beyond wave breaking by considering pairs $(u,\mu)$, where $u$ is the wave profile, while $\mu$ denotes the positive Radon measure describing the energy distribution and satisfying $d\mu_{\mathrm{ac}}=(u^2+u_x^2) dx$. The two most prominent solution concepts are the conservative solutions \cite{BC2, HR2, EK}, which preserve the energy $\mu(t,\Real)$ and the dissipative solutions \cite{BC, HR, GHR}, which lose all the energy which concentrates on sets of measure zero and is described by $\mu_{\mathrm{sing}}$, the singular part of $\mu$. An interpolation between these two is given by the $\alpha$-dissipative solutions with $\alpha\in [0,1]$, where an $\alpha$ part of the concentrated energy is dissipated upon wave breaking \cite{GHR2}. In addition, there exist also solutions, which are not covered by the concept of $\alpha$-dissipative solutions like the stumpons, which are traveling wave solutions, but not conservative, see \cite{GG}.
 
In this article we focus on studying the uniqueness of weak dissipative solutions to the CH equation. The existence of these solutions has been shown, see e.g., \cite{BC,HR,GHR2}, by rewriting the CH equation as a system of ordinary differential equations via a generalized method of characteristics. The underlying idea is the following. Applying the inverse Helmholtz operator to \eqref{CH}, the CH equation rewrites as
\begin{equation}\label{CH1}
 u_t+uu_x= -p_x,
 \end{equation}
 where
 \begin{equation*}
 p(t,x)= \frac14 \int_\Real e^{-\vert x-y\vert } (2u^2+u_x^2)(t,y) dy
 \end{equation*}
 and $u(t,\cdot)$ belongs to $H^1(\Real)$.  
 Under the assumption that $u$ is a smooth solution in $H^1(\Real)$, the time evolution of $(u^2+u_x^2)(t,x)$ is given by 
 \begin{equation}\label{CH2}
 (u^2+u_x^2)_t + (u(u^2+u_x^2))_x= (u^3-2pu)_x,
 \end{equation}
 and the solution to \eqref{CH1} and \eqref{CH2} can be computed by the classical method of characteristics and is unique. Here, the characteristic $y(t,\xi)$ satisfies the ordinary differential equation 
 \begin{equation}\label{int:char}
 y_t(t,\xi)=u(t,y(t,\xi)).
 \end{equation}
 In the case of weak solutions to the CH equation, due to wave breaking, $u(t,\cdot)$ can only be expected to be H{\"o}lder continuous and hence the above differential equation might not have a unique solution. Furthermore, $(u^2+u_x^2)(t,x)$ turns into a positive, finite Radon measure $d\mu$, with $d\mu_{\mathrm{ac}}= (u^2+u_x^2)dx$ as wave breaking occurs. In the case of conservative solutions, cf. \cite{BC2,HR2}, this measure would satisfy the transport equation 
 \begin{equation*}
 \mu_t+(u\mu)_x= (u^3-2pu)_x,
 \end{equation*}
 while in the dissipative case one has 
 \begin{equation*}
 \mu_t+(u\mu)_x\leq (u^3-2pu)_x.
 \end{equation*}
 However, the latter inequality does not specify how much energy is dissipated upon wave breaking and in fact there is no (partial) differential equation known, which $\mu$ satisfies. Thus to specify how much energy is exactly dissipated upon wave breaking, the equation is reformulated via a generalized method of characteristics in Lagrangian coordinates, where it is possible to define the dissipative as well as the $\alpha$-dissipative solution concept for any $\alpha\in [0,1]$, see \cite{GHR2}. 
 
For the dissipative solutions, the key idea behind the underlying change of variables is to consider triplets $(u,\mu,\nu)= (u, (u^2+u_x^2), \nu)$ instead of pairs $(u,\mu)= (u,(u^2+u_x^2) dx)$, where $u$ satisfies \eqref{CH1} and $\nu$ satisfies the transport equation 
 \begin{equation}\label{int:trans}
\nu_t+(u\nu)_x= (u^3-2pu)_x,
\end{equation}
together with $d\nu_{ac}= (u^2+u_x^2) dx$. Here, $\nu(0)$, the measure at time $t=0$, can be any positive, finite Radon measure such that $d\nu_{ac}(0)=(u^2_0+u_{0,x}^2) dx$ and hence is not unique. Nevertheless, the choice of $\nu$ has no influence on the resulting pair $(u,\mu)$, cf. Lemma~\ref{lem:main2} and for any $t$ one has $d\nu_{ac}(t)= (u^2+u_x^2)(t) dx$. Furthermore, the system of equations \eqref{CH1} and \eqref{int:trans} can be studied using a generalized method of characteristics from \cite{GHR2}, which we sketch next.

To any triplet $(u,\mu,\nu)$ we can associate a quadruple of Lagrangian coordinates $(y,U,V,H)$, where $y(t,\xi)$ denotes the characteristics, $U(t,\xi)=u(t,y(t,\xi))$ the solution along the characteristics, $V_\xi(t,\xi)$ the energy of the particle $\xi$ at time $t$ and $H(t,\xi)$ is related to $\nu$. Introducing in addition, $P(t,\xi)=p(t,y(t,\xi))$ and $Q(t,\xi)= p_x(t,y(t,\xi))$, the time evolution of the Lagrangian variables is then governed by the following system of differential equations 
\begin{align}\label{int:diff}
y_t(t,\xi)&= U(t,\xi),\\ 
U_t(t,\xi)&=-Q(t,\xi),\\ 
H_t(t,\xi)&= (U^3-2PU)(t,\xi),\\ 
V(t,\xi)&= \int_{-\infty }^\xi  (1- \mathbbm{1}_{\{\zeta\mid \tau(\zeta)\leq t\}} (\eta))H_\xi(t,\eta)d\eta,
\end{align}
where
\begin{equation*}\tau(\xi)= \begin{cases} 0, &\quad y_\xi(0,\xi)=0,\\
\sup\{t\in \Real^+ \mid y_\xi(t',\xi)>0 \text{ for all } 0\leq t'<t\}, & \quad \text{ otherwise,}
\end{cases}
\end{equation*}
which is uniquely solvable. 

The loss of the energy is encoded in the variable $V_\xi(t,\xi)$. Indeed, along a characteristic labeled with $\xi$, wave breaking occurs at time $\tau(\xi)$, i.e., when $y_\xi(t, \xi)$ becomes zero. When this happens the energy of the particle $\xi$, given by $V_\xi(t,\xi)$, drops to zero and remains zero thereafter. Thus at wave breaking the maximal amount of energy is dissipated, which is why one also characterizes the dissipative solutions as those solutions which dissipate energy at the fastest possible rate. 

To prove the uniqueness of dissipative solutions the starting points is to identify properties and constraints, which are solely satisfied by this class of solutions. Beside of $u$ being a weak solution, the two most prominent properties are 
\begin{itemize}
\item $u$ satisfies a one-sided Lipschitz condition,
\item $\norm{u(t,\cdot)}_{H^1(\Real)}$ is non-increasing.
\end{itemize}
In \cite[Proposition 6.7]{HR} it has been proven that for almost every $x$ and all $t\geq 0$
\begin{equation}
u_x(t,x)\leq \frac{2}{t}+\sqrt{2}\norm{u_0}_{H^1(\Real)}.
\end{equation}
As a consequence the differential equation \eqref{int:char} can be solved uniquely backwards in time. Assuming that the backwards characteristics satisfy $y(T,\xi)=\xi$ for $\xi \in \Real$, it cannot be expected that 
\begin{equation*}
\{(t,y(t,\xi)) \mid (t,\xi)\in [0,T]\times \Real\}
\end{equation*}
covers all of $[0,T]\times \Real$, see e.g. \cite[Chapter 4]{BJ}. This phenomenon is closely related to the loss of energy. In particular, following characteristics backward in time has been used in \cite{BJ} to prove the uniqueness of weak solutions to transport equations like \eqref{int:trans}, where $u$ satisfies a one-sided Lipschitz condition. A result which also motivates why combining \eqref{CH1} and \eqref{int:trans} and introducing equivalence classes with respect to $\nu$ is favorable. On the other hand, inspired by \cite{D}, the authors in \cite{J} showed by following dissipative solutions backward along characteristics that they coincide with the solutions constructed in \cite{BC} and hence they must be unique. The argument is solely based on $u$ and \eqref{CH1} and does not involve measures $\nu$ at all, which complicates the proof.

In this work, we again prove the uniqueness of dissipative solutions to the CH equation, but our argumentation is not based on following solutions backward along characteristics. Instead, we use, inspired by \cite{BCZ, GH}, a carefully selected change of variables, which forces us to introduce equivalence classes with respect to $\nu$, together with a more exhaustive list of constraints, which relies on a good understanding of the dissipative solutions constructed in \cite{GHR2}, see Section \ref{sec:back}. While this seems to complicate the proof of the uniqueness of dissipative solutions at first glance, the opposite is the case. Since the dissipation is very well encoded in the measure differential equation for $\nu$, \eqref{int:trans}, we can immediately define a change of variables, which allows to deduce that each dissipative solution must satisfy \eqref{int:diff}, which is uniquely solvable, see Theorem~\ref{main} in Section~\ref{sec:unique}. Thereafter, in Section~\ref{subsec:equivcl}, it is shown that the wave profile $u$ is independent of $\nu$, see Lemma~\ref{lem:main2}, and hence unique. 

Finally, we want to point out that this approach, with slight modifications, can also be used to prove the uniqueness of dissipative solutions to the Hunter--Saxton (HS) equation. Another very elegant proof for the HS equation, which inspired \cite{J}, can be found in \cite{D}. 

\section{Characterizing dissipative solutions}\label{sec:back}
 
In this section, we introduce the concept of weak dissipative solutions $u$ for the Camassa--Holm equation. To be more precise, we will characterize dissipative solutions $u$ with the help of equivalence classes for pairs $(u,\nu)$. 
Thereafter, we will show that such pairs $(u,\nu)$ exist.

The following set serves as a basis for introducing equivalence classes.

\begin{definition}[Eulerian coordinates]\label{def:Euler}
The space $\D$ consists of all pairs $(u,\nu)$ such that 
\begin{itemize}
\item $u\in H^1_u(\Real)$,
\item $\nu\in\mathcal{M}^+(\Real)$,
\item $d\nu_{ac}=(u^2+u_x^2) dx$,
\end{itemize}
where $\mathcal{M}^+(\Real)$ denotes the set of positive, finite Radon measures on $\mathbb{R}$ and 
\begin{equation*}
H^1_u(\Real)=\{f\in H^1(\Real)\mid \text{there exists } D\in \Real \text{ such that } f'(x)\leq D \text{ for a.e. } x\in \Real\}.
\end{equation*}
\end{definition}

The measure $\nu$ is a {\it dummy} measure. By this we mean that it contains no additional information about the solution $u$. Therefore the choice of $\nu$ should have no influence on the solution $u$, but gives rise to equivalence classes.

\begin{definition}[Equivalence relation]\label{def:equivEuler}
We say that two elements $(u_1, \nu_1)$ and $(u_2, \nu_2)$ in $\D$ belong to the same equivalence class if $u_1=u_2$.
\end{definition}

As a consequence also the definition of weak dissipative solutions $u$ has to take these equivalence classes into account. We therefore introduce first all constraints, which are necessary and sufficient to guarantee the existence of pairs $(u, \nu)$, which are weak dissipative solutions to the Camassa--Holm equation in the following sense. 

\begin{definition}\label{def:dissol2}
We say that $(u,\nu)$ is a global weak dissipative solution of the Camassa--Holm equation with initial data $(u(0), \nu(0))\in \D$, if 
\begin{enumerate}
\item For each fixed $t\geq 0$ we have  $(u(t),\nu(t))\in \D$.
\item \label{cond:dissol2:2} For any $0\leq T_1<T_2$, the function $u:[T_1,T_2]\times \Real\to \Real$ is H{\"o}lder continuous with H{\"o}lder exponent one-half and the map $t\mapsto u(t,\cdot)$ is Lipschitz continuous from $[T_1,T_2]$ into $L^2(\Real)$. 
\item \label{cond:dissol2:22}There exists a constant $D>0 $ such that for each $t\geq 0$
\begin{equation}\label{oneside:Lip}
u_x(t,x)\leq D  \quad  \text{for almost every } x\in \Real.
\end{equation}
\item \label{cond:dissol2:3} The function $F(t,x): \Real^+\times \Real\to \Real$, given by 
\begin{equation}\label{def:F}
F(t,x)=\int_{-\infty}^x (u^2+u^2_x)(t,y)dy,
\end{equation}
satisfies
\begin{equation}\label{prop:lim:F}
F(s,y)\to F(t,x) \quad \text{ for all } (s,y)\to (t,x) \text{ such that } s\geq t.
\end{equation}
\item \label{cond:dissol2:4}The function $p(t,x):\Real^+\times \Real \to \Real$, given by 
\begin{equation}\label{def:p}
p(t,x)=\frac14 \int_{\mathbb{R}} e^{-\vert x-y\vert} (2u^2+u_x^2) (t,y) dy,
\end{equation}
satisfies 
\begin{equation}\label{cond:limp}
p(s,y)\to p(t,x) \quad \text{ for all } (s,y) \to (t,x) \text{ such that } s\geq t
\end{equation}
and 
\begin{equation}\label{cond:limpx}
p_x(s,y)\to p_x(t,x) \quad \text{ for all } (s,y) \to (t,x) \text{ such that } s\geq t.
\end{equation}
\item \label{cond:dissol2:5} For any $0\leq T_1< T_2$, the function $p(t,\sigma(t))$, where $\sigma:[T_1, T_2]\to\Real$ is Lipschitz continuous with Lipschitz constant $\hat L$, is a function of bounded variation on $[T_1, T_2]$ and 
\begin{equation*}
T.V. (p(t,\sigma(t)))\leq B(\hat L),
\end{equation*}
where $B(\hat L)$ denotes a constant which is only dependent on $\hat L$.
\item \label{cond:dissol2:7} For any $0\leq T_1< T_2$, the function $u$ satisfies for any $\phi\in C_c^\infty([T_1, T_2]\times\Real)$
\begin{equation}\label{weak:u}
\int_{T_1}^{T_2} \int_\Real \big(u\phi_t+\frac12 u^2\phi_x-p_x\phi\big)(t,x)dx dt=\int_{\Real} u\phi(T_2, x) dx-\int_\Real u\phi(T_1,x) dx.
\end{equation} 
\item \label{cond:dissol2:8} For any $0\leq T_1<T_2$, the pair $(u,\nu)$ satisfies for any $\phi\in C_c^\infty ([T_1,T_2]\times \Real)$
\begin{align}\nonumber
\int_{T_1}^{T_2} \int_\Real (\phi_t+u\phi_x)(t,x)d\nu(t)dt& -\int_{T_1}^{T_2} \int_\Real (u^3-2pu)\phi_x(t,x) dx dt \\ \label{weak:trans}
&\qquad\quad   = \int_\Real \phi(T_2,x)d\nu(T_2)-\int_\Real \phi(T_1,x)d\nu(T_1).
\end{align}
\item The family of Radon measures $\{\nu(t)\mid t\in \Real^+\}$ depends continuously on time w.r.t. the topology of weak convergence of measures. 
\end{enumerate}
\end{definition}

Note that $\{\nu(t)\mid t\in \Real^+\}$ provides a measure valued solution $w$ to the linear transport equation 
\begin{equation*}
w_t+(uw)_x=(u^3-2pu)_x.
\end{equation*}
Since $u(t,\cdot)$ and hence also $p(t,\cdot)$ belong to $H^1(\Real)$ for all $t\geq0$, one has $\nu(t,\Real)=\nu(0,\Real)$ for all $t\geq 0$. Furthermore, 
\begin{equation}\label{est:uinf}
\norm{u(t,\cdot)}_\infty^2\leq \int_\Real (u^2+u_x^2)(t,x)dx \leq \nu(t,\Real)= \nu(0,\Real)
\end{equation}
and  
\begin{equation}\label{est:pxpinf}
\norm{p_x(t,\cdot)}_\infty\leq \norm{p(t,\cdot)}_\infty\leq  \nu(t,\Real)= \nu(0,\Real).
\end{equation}

In \cite{GHR2}, weak dissipative solutions $(u,\nu)$ have been constructed. A closer look at the construction therein reveals that the following theorem holds.

\begin{theorem}
For any initial data $(u(0),\nu(0))\in \D$, the Camassa--Holm equation has a global weak dissipative solution $(u,\nu)$ in the sense of Definition~\ref{def:dissol2}.
\end{theorem}

In other words, all the properties stated in Definition~\ref{def:dissol2} are satisfied by the dissipative solutions constructed in \cite{GHR2}. However, some of them are better hidden than others. This is especially true for Definition~\ref{def:dissol2} \eqref{cond:dissol2:22}--\eqref{cond:dissol2:5}, which we will show at the end of this Section.  

Finally, we can define weak dissipative solutions $u$ to the Camassa--Holm equation by using the equivalence relation from Definition~\ref{def:equivEuler}. 

\begin{definition}\label{def:dissol}
We say that $u$ is a weak dissipative solution of the Camassa--Holm equation with initial data $u(0)\in H^1_u(\Real)$ if 
\begin{enumerate}
\item $u(t,\cdot)\in H^1_u(\Real)$ for all $t\geq 0$.
\item\label{cond:dissol:2} For any $(u_1(0), \nu_1(0))$ and $(u_2(0), \nu_2(0))$ in $\D$, such that 
\begin{equation*}
u_1(0)=u(0)=u_2(0), 
\end{equation*}
the corresponding weak solutions $(u_1, \nu_1)$ and $(u_2, \nu_2)$ satisfy
\begin{equation*}
u_1(t)=u(t)=u_2(t) \quad \text{ for all } t \geq 0.
\end{equation*}
 \end{enumerate} 
\end{definition}

With the definition of a weak dissipative solution $u$ in place, we can turn our attention to showing that the dissipative solutions $(u,\nu)$ constructed in \cite{GHR2} satisfy Definition~\ref{def:dissol2} \eqref{cond:dissol2:22}--\eqref{cond:dissol2:5}. 

We start by outlining the construction of weak dissipative solutions from \cite{GHR2}, which is based on a generalized method of characteristics and hence involves a change of variables from Eulerian coordinates $(u,(u^2+u_x^2)dx, \nu)$ to Lagrangian coordinates $(y,U,V,H)$.

\begin{definition}[Lagrangian coordinates]\label{def:Lagcoord} The set $\F$ consists of all tuples $(y,U,V,H)$ such that 
\begin{enumerate}
\item $U\in L^2(\Real)$,
\item $(y-\id, U, V, H)\in [W^{1,\infty}(\Real)]^4$,
\item $(y_\xi-1, U_\xi, V_\xi, H_\xi)\in [L^2(\Real)]^4$,
\item $\displaystyle\lim_{\xi\to -\infty} V(\xi)=0= \lim_{\xi\to -\infty} H(\xi)$,
\item $y_\xi\geq 0$, $H_\xi\geq V_\xi\geq 0$ a.e.,
\item there exists $c>0$ such that $y_\xi+H_\xi\geq c>0$ a.e.,
\item \label{cond:Lagcoord7} $U^2y_\xi^2+U_\xi^2=y_\xi V_\xi$ a.e., 
\item $y_\xi(\xi)=0$ implies $V_\xi(\xi)=0$ a.e..
\end{enumerate}
\end{definition}

Note that there cannot be a one-to-one correspondence between Eulerian and Lagrangian coordinates, since we have triplets $(u, (u^2+u_x^2)dx, \nu)$ which are related to quadruplets $(y,U,V,H)$. Instead, each element in Eulerian coordinates corresponds to a whole equivalence class of elements in Lagrangian coordinates. These equivalence classes can be identified using so-called relabeling functions.

\begin{definition}[Relabeling functions]\label{def:rel} We denote by $\G$ the group of homeomorphisms $f$ from $\Real$ to $\Real$ such that 
\begin{align*}
f-\id \text{ and }& f^{-1}-\id \text{ both belong to } W^{1,\infty}(\Real),\\
&f_\xi-1 \text{ belongs to } L^2(\Real),
\end{align*}
where $\id$ denotes the identity function.
\end{definition}

Whether or not a function is a relabeling function can be checked using the following lemma, which is taken from \cite{HR2}. 

\begin{lemma}[Identifying relabeling functions]\label{rellem} If $f$ is absolutely continuous, $f-\id\in W^{1,\infty}(\Real)$, $f_\xi-1\in L^2(\Real)$, and there exists $c\geq 1$ such that $\frac{1}{c}\leq f_\xi\leq c$ almost everywhere, then $f\in \G$.
\end{lemma}

Let $X_1=(y_1, U_1, V_1, H_1)$ and $X_2=(y_2, U_2, V_2, H_2)$ in $\F$. Then $X_1$ and $X_2$ belong to the same equivalence class if there exists a relabeling function $f\in \G$ such that 
\begin{equation*}
X_1\circ f=(y_1\circ f, U_1\circ f, V_1\circ f, H_1\circ f)=(y_2, U_2, V_2, H_2)=X_2.
\end{equation*}
Furthermore, let 
\begin{equation*}
\F_0=\{(y,U,V,H)\in \F\mid y+H=\id\}.
\end{equation*}
Then $\F_0$ contains exactly one representative of each equivalence class in $\F$. Moreover, one has that if $X=(y,U,V,H)\in \F_0$ and $f\in \G$, then 
\begin{equation*}
y\circ f+H\circ f= f.
\end{equation*}
This implies that for each $X=(y,U,V,H)\in \F$, one has $y+H\in \G$.

In \cite{GHR2}, one rewrites the Camassa--Holm equation in Lagrangian coordinates as the following semilinear system of differential equations with discontinuous right hand side,
\begin{subequations}\label{sys:diffeq}
\begin{align}\label{sys:diffeq1}
y_t(t,\xi)&= U(t,\xi),\\ \label{sys:diffeq2}
U_t(t,\xi)&=-Q(t,\xi),\\ \label{sys:diffeq3}
H_t(t,\xi)&= (U^3-2PU)(t,\xi),\\ \label{sys:diffeq4}
V(t,\xi)&= \int_{-\infty }^\xi  (1- \mathbbm{1}_{\{\zeta\mid \tau(\zeta)\leq t\}} (\eta))H_\xi(t,\eta)d\eta,
\end{align}
\end{subequations}
where
\begin{equation}\label{def:tau}
\tau(\xi)= \begin{cases} 0, &\quad y_\xi(0,\xi)=0,\\
\sup\{t\in \Real^+ \mid y_\xi(t',\xi)>0 \text{ for all } 0\leq t'<t\}, & \quad \text{ otherwise,}
\end{cases}
\end{equation}
\begin{equation}\label{def:Pb}
P(t,\xi)= \frac14 \int_{\Real} e^{-\vert y(t,\xi)-y(t,\eta)\vert} (U^2y_\xi+V_\xi)(t,\eta)d \eta,
\end{equation}
and
\begin{equation}\label{def:Qb}
Q(t,\xi)=-\frac14 \int_{\Real} \mathrm{sign}(\xi-\eta)e^{-\vert y(t,\xi)-y(t,\eta)\vert}(U^2y_\xi+V_\xi)(t,\eta) d\eta,
\end{equation}
respectively.

The function $\tau(\xi)$ in the above system associates to each $\xi$ the time at which wave breaking occurs along the characteristic $y(t,\xi)$. The function $V_\xi(t,\xi)$, on the other hand, corresponds to the energy of the particle $\xi$ at time $t$, which by definition vanishes from the system at time $\tau(\xi)$, as the differentiated version of \eqref{sys:diffeq} reveals, i.e., 
\begin{subequations}\label{sys:diffeqxi}
\begin{align} \label{sys:diffeqxi1}
y_{\xi,t}(t,\xi)&= U_\xi(t,\xi),\\\label{sys:diffeqxi2}
U_{\xi,t}(t,\xi)&= \frac12 (V_\xi+(U^2-2P)y_\xi)(t,\xi),\\ \label{sys:diffeqxi3}
H_{\xi,t}(t,\xi)&=((3U^2-2P)U_\xi-2QUy_\xi)(t,\xi),\\ \label{sys:diffeqxi4}
V_\xi(t,\xi)&= (1-\mathbbm{1}_{\{\zeta\mid \tau(\zeta)\leq t\}}(\xi))H_\xi(t,\xi).
\end{align}
\end{subequations}

That the system \eqref{sys:diffeq}--\eqref{sys:diffeqxi} is uniquely solvable in $\F$ has been shown, in \cite{GHR2}, by considering the corresponding system of integral equations. Important properties of the solutions, which play a key role for the remainder of this section, are the Lagrangian formulations of \eqref{est:uinf} and \eqref{est:pxpinf}, i.e.,
\begin{equation}\label{H:ind}
\lim_{\xi\to\infty}H(t,\xi)=\lim_{\xi\to \infty}H(0,\xi)=H_\infty,
\end{equation}
and there exists a constant $C>0$, only dependent on $H_\infty$, such that 
\begin{equation}\label{Lag:C}
\norm{U(t,\cdot)}_\infty, \norm{Q(t,\cdot)}_\infty, \norm{P(t,\cdot)}_\infty\leq C \quad \text{ for all } t\geq 0.
\end{equation}

Finally it remains to introduce the mappings between Eulerian and Lagrangian coordinates.

\begin{definition}\label{def:EultoLag}
Let the mapping $L:\D\to \F_0$ be defined by $L((u,\nu))= (y,U,V,H)$, where 
\begin{align*}
y(\xi)&=\sup\{x\mid x+\nu((-\infty,x))<\xi\},\\
H(\xi)&=\xi-y(\xi),\\
U(\xi)&= u(y(\xi)),\\
V(\xi)& = \int_{-\infty}^{y(\xi)} (u^2+u_x^2)(z) dz= \int_{-\infty }^{\xi } \mathbbm{1}_{\{\zeta\mid y_\xi(\zeta)\not=0\}}H_\xi(\eta)  d\eta.
\end{align*}
\end{definition}

\begin{definition}
Let the mapping $M:\F\to \D$ be defined by $M((y,U,V,H))= (u,\nu)$, where
\begin{align*}
u(x)=U(\xi) & \text{ for any } \xi\text{ such that }x=y(\xi),\\
(u^2+u_x^2) dx& = y_{\#}(V_\xi d\xi),\\
\nu&= y_{\#}(H_\xi d\xi).
\end{align*}
\end{definition}

Now we are ready to show that the weak dissipative solutions constructed in \cite{GHR2} satisfy Definition~\ref{def:dissol2} \eqref{cond:dissol2:22}--\eqref{cond:dissol2:5}.

\subsection{$u_x$ satisfies Definition~\ref{def:dissol2} \eqref{cond:dissol2:22}}
For $ t\geq 0$ define
\begin{equation*}
B(t)=\{\xi \in \Real \mid y_\xi( t, \xi), U_\xi( t, \xi) \text{ are differentiable and } y_\xi( t, \xi)>0\}. 
\end{equation*}
Then $y( t,B(t))$ has full measure and for almost every $x\in \Real$ there exists $\xi$ such that $y(t,\xi)=x$, $y_\xi( t,\xi)>0$ and $U_\xi( t,\xi)=u_x( t,y( t,\xi))y_\xi( t,\xi)$. Furthermore, a closer look at the system \eqref{sys:diffeqxi} reveals that $y_\xi(\tilde t,\xi)=0$ for some $(\tilde t,\xi)$, implies $y_\xi(t,\xi)=U_\xi(t,\xi)=V_\xi(t,\xi)=0$ and $H_\xi(t,\xi)=H_\xi(\tilde t, \xi)$ for all $t\geq \tilde t$. That means, if $\xi \in B(0)$, then the function
\begin{equation*}
\alpha(t,\xi)=u_x(t,y(t,\xi))= \frac{U_\xi(t,\xi)}{y_\xi(t,\xi)}
\end{equation*}
is well-defined for all $t$ such that $0\leq t<\tau(\xi)$. In addition, $\alpha$ satisfies the differential equation
\begin{equation}\label{ux:diff}
\alpha_t+\frac12 \alpha^2=(U^2-P),
\end{equation}
where the right hand side can be uniformly bounded by a constant $C>0$, cf. \eqref{Lag:C}. 
Thus 
\begin{equation*}
\alpha_t \leq  -\frac14 \alpha^2+C= (\sqrt{C}-\frac12\vert \alpha\vert )(\sqrt{C}+ \frac12\vert \alpha\vert ),
\end{equation*}
where the right hand side is negative for $\vert \alpha\vert > 2\sqrt{C}$ and therefore any solution to \eqref{ux:diff}, must satisfy 
\begin{equation*}
\alpha(t)\leq \max (\alpha(0),  2\sqrt{C}).
\end{equation*}

\subsection{$F$ satisfies Definition~\ref{def:dissol2} \eqref{cond:dissol2:3}} By definition $F(t,x)$ satisfies 
\begin{equation}\label{rel:FV}
F(t,y(t,\xi))=V(t,\xi)= \int_{-\infty}^\xi V_\xi(t,\eta)d\eta
\end{equation}
and by \eqref{sys:diffeqxi}
\begin{equation}\label{lim:Vxi}
\lim_{s\downarrow t} V_\xi(s,\xi)=V_\xi(t,\xi) \quad \text{ for a.e. } \xi \in \Real.
\end{equation}
Furthermore, it can be established, using Definition~\ref{def:Lagcoord} and \eqref{sys:diffeqxi}, that there exists a function $g\in L^1(\Real)$ such that for every $l\in [t,s]$
\begin{equation*}
0\leq V_\xi(l,\xi)\leq H_\xi(l,\xi)\leq g(\xi) \quad \text{ for almost every } \xi \in \Real.
\end{equation*}
Thus the dominated convergence theorem yields 
\begin{equation*}
F(s, y(s,\xi))\to F(t,y(t,\xi)) \quad \text{ for all } s\to t \text{ such that } s\geq t.
\end{equation*}

Since $F(t,\cdot)$ is absolutely continuous for every $t$ and $y(t,\xi)$ is continuous with respect to both space and time, \eqref{prop:lim:F} follows.

\subsection{$p$ and $p_x$ satisfy Definition~\ref{def:dissol2} \eqref{cond:dissol2:4}}
Comparing \eqref{def:p}, \eqref{def:Pb}, and \eqref{def:Qb}, we observe that 
\begin{equation}\label{rel:pPpxQ}
p(t,y(t,\xi))=P(t,\xi) \quad \text{ and } \quad p_x(t,y(t,\xi))=Q(t,\xi),
\end{equation} 
which implies, since $y(t,\xi)$, $y_\xi(t,\xi)$, and $U(t,\xi)$ are continuous with respect to time and \eqref{lim:Vxi} holds, that 
\begin{equation}\label{lim:p2}
p(s,y(s,\xi))\to p(t,y(t,\xi)) \quad \text{and} \quad p_x(s,y(s,\xi))\to p_x(t,y(t,\xi))
\end{equation}
for all $s\to t$ such that $s\geq t$.

Furthermore, $y(t,\xi)$ is continuous with respect to time and space and hence \eqref{cond:limp} and \eqref{cond:limpx} will be satisfied, if we can show that for each $t\geq 0$ the functions $p(t,\cdot)$ and $p_x(t,\cdot)$ are continuous.

The function $p(t,\cdot)$, given by \eqref{def:p}, belongs to $H^1(\mathbb{R})$ and hence is continuous. As far as $p_x(t,\cdot)$ is concerned, observe that $u(t,\cdot)\in H^1(\Real)$ and \eqref{prop:lim:F} imply that the function
\begin{equation*}
\hat F(t, x)= \int_{-\infty}^x (2u^2+u_x^2) (t,y) dy= \int_{-\infty}^x u^2(t,y) dy + F(t,x)
\end{equation*}
is uniformly bounded and continuous with respect to $x$ for each fixed $t\geq 0$. Thus rewriting $p_x(t,x)$ as 
\begin{align*}
p_x(t,x)& = -\frac12 \int_{-\infty}^x e^{y-x}(2u^2+u_x^2) (t,y) dy+ p(t,x)\\
&= -\frac12 \hat F(t,x)+\frac12 \int_{-\infty}^x e^{y-x} \hat F(t,y)dy+p(t,x)
\end{align*}
finishes the proof of \eqref{cond:limpx}.

\subsection{$p$ satisfies Definition~\ref{def:dissol2} \eqref{cond:dissol2:5}}
Given a Lipschitz continuous curve $\sigma:[T_1,T_2]\to \Real$ with Lipschitz constant $\hat L$, we must show that 
\begin{equation*}
T.V. (p(t,\sigma(t))):= \sup \sum_{i} \vert p(t_i, \sigma(t_i))-p(t_{i-1}, \sigma(t_{i-1}))\vert,
\end{equation*}
where the supremum is taken over all finite partitions $\{t_i\}$ of $[T_1, T_2]$ such that $t_i<t_{i+1}$, is finite.

Let $\{t_i\}$ be a finite partition of $[T_1, T_2]$. Then to every $i$ there exists a $\xi_i\in \Real$, not necessarily unique, such that $(t_i, \sigma(t_i))=(t_i, y(t_i, \xi_i))$ and using the triangle inequality, we have
\begin{align*}
\sum_i \vert p(t_i, \sigma(t_i))-p(t_{i-1}, \sigma(t_{i-1}))\vert 
& \leq \sum_i \vert p(t_i, y(t_i, \xi_i))-p(t_{i-1}, y(t_{i-1}, \xi_i))\vert\\
& \quad  + \sum_i \vert p(t_{i-1}, y(t_{i-1}, \xi_i))-p(t_{i-1}, y(t_i, \xi_i))\vert\\
& \quad + \sum_{i} \vert p(t_{i-1}, \sigma(t_i))-p(t_{i-1}, \sigma(t_{i-1}))\vert\\
& = S_1+S_2+S_3.
\end{align*}

When deriving next an upper bound for each of the sums $S_1$, $S_2$, and $S_3$, which is independent of the partition $\{t_i\}$, \eqref{sys:diffeq}, \eqref{sys:diffeqxi}, and \eqref{Lag:C} will play a major role. In addition, we will from now on denote by $C$ positive constants, which are independent of $t$ and  which might change from line to line.

{\it Estimate for $S_1$:}  We split $p(t,y(t,\xi))$ as follows 
\begin{align*}
p(t,y(t,\xi))&= \frac14 \int_\Real e^{-\vert y(t,\xi)-y(t,\eta)\vert} U^2y_\xi(t,\eta) d\eta+ \frac14 \int_\Real e^{-\vert y(t,\xi)-y(t,\eta)\vert} V_\xi(t,\eta) d\eta\\
& =P_1(t,\xi)+P_2(t,\xi).
\end{align*}

To estimate $P_1$, observe that for $s\geq t$, by \eqref{sys:diffeq1} and \eqref{Lag:C},
\begin{equation}\label{est:yts}
\vert y(s,\xi)-y(t,\xi)\vert \leq \int_t^s \vert U(l,\xi)\vert dl\leq C\vert s-t\vert.
\end{equation}
Furthermore, Definition~\ref{def:Lagcoord} \eqref{cond:Lagcoord7} combined with \eqref{sys:diffeqxi} implies that 
\begin{equation*}
(y_\xi+H_\xi)_t(t,\xi)\leq C(y_\xi+H_\xi)(t,\xi),
\end{equation*}
and, applying Definition~\ref{def:Lagcoord} once more, we have for all $s\geq t$
\begin{equation}\label{est:growHxyx}
\vert U_\xi(s,\xi)\vert \leq(y_\xi+V_\xi)(s,\xi)\leq (y_\xi+H_\xi)(s,\xi)\leq e^{C(s-t)}(y_\xi+H_\xi)(t,\xi).
\end{equation}
Thus, using \eqref{Lag:C} and \eqref{est:growHxyx}, 
\begin{align*}
\vert U^2y_\xi(s,\xi)-U^2y_\xi(t,\xi)\vert &\leq \int_t^s \vert -2QUy_\xi+U^2U_\xi\vert (l,\xi) dl\\
& \leq C\int_t^s (y_\xi+H_\xi)(l, \xi )dl\\
& \leq Ce^{ C(s-t)} (y_\xi+H_\xi)(t, \xi) \vert s-t\vert.
\end{align*}
Recalling \eqref{est:yts}, we therefore end up with 
\begin{align}\nonumber
\vert P_1(s,\xi)-P_1(t,\xi)\vert& \leq \frac14 \int_\Real \vert e^{-\vert y(s,\xi)-y(s,\eta)\vert } - e^{-\vert y(t,\xi)-y(t,\eta)\vert}\vert U^2y_\xi(s,\eta) d\eta\\ \nonumber
& \quad + \frac14\int_\Real e^{-\vert y(t,\xi)-y(t,\eta)\vert} \vert U^2y_\xi(s,\eta)-U^2y_\xi(t,\eta)\vert d\eta\\ \nonumber
& \leq  C \int_\Real U^2y_\xi(s,\eta)d\eta\vert s-t\vert\\ \nonumber
& \quad +  C e^{ C(T_2-T_1)} \int_\Real e^{-\vert y(t,\xi)-y(t,\eta)\vert} (y_\xi+H_\xi)(t, \eta ) d\eta\vert s-t\vert\\ \nonumber
& \leq  C H_\infty \vert s-t\vert+ C e^{C(T_2-T_1)}\left(1+ H_\infty \right)\vert s-t\vert \\ \label{est:P1}
& \leq Ce^{C(T_2-T_1)}\vert s-t\vert.
\end{align}
Here we also used Definition~\ref{def:Lagcoord} \eqref{cond:Lagcoord7} and \eqref{H:ind}.

For $P_2$, the following inequality, which holds for any $s\geq t$, plays a key role
\begin{equation*}
\vert V_\xi(t,\xi)-V_\xi(s,\xi)\vert \leq \vert H_\xi(t,\xi)-H_\xi(s,\xi)\vert + (y_\xi+H_\xi)(s,\xi) \mathbbm{1}_{\{\zeta\mid t<\tau(\zeta)\leq s\}}(\xi).
\end{equation*}
Combined with \eqref{sys:diffeqxi3}, \eqref{Lag:C}, \eqref{est:yts}, and \eqref{est:growHxyx}, the last inequality implies that
\begin{align*}
\vert P_2(s,\xi)-P_2(t,\xi)\vert&\leq \frac14 \int_\Real \vert  e^{-\vert y(s,\xi)-y(s,\eta)\vert} - e^{-\vert y(t,\xi)-y(t,\eta)\vert} \vert H_\xi(s,\eta) d\eta\\
& \quad +\frac14\int_{\Real } e^{-\vert y(t,\xi)-y(t,\eta)\vert} \vert H_\xi(s,\eta)-H_\xi(t,\eta) \vert d\eta\\
& \quad +\frac14 \int_{\Real } e^{-\vert y(t,\xi)-y(t,\eta)\vert }(y_\xi+H_\xi)(s,\eta) \mathbbm{1}_{\{\zeta \mid t<\tau(\zeta)\leq s\}}(\eta) d\eta\\
& \leq  C \int_{\Real} H_\xi(s,\eta) d\eta\vert s-t\vert \\
& \quad + Ce^{C(s-t)} \int_\Real e^{-\vert y(t,\xi)-y(t,\eta)\vert} (y_\xi+H_\xi)(t,\eta)d\eta \vert s-t\vert \\
& \quad + e^{C(t-T_1)} \int_{\Real} e^{- \vert y(T_1, \xi)-y(T_1, \eta)\vert}(y_\xi+H_\xi)(s, \eta)\\
&\qquad  \qquad \quad \qquad \qquad \qquad\qquad   \times \mathbbm{1}_{\{\zeta \mid t<\tau(\zeta)\leq s\}}(\eta) d\eta\\
& \leq Ce^{C(T_2-T_1)}\vert s-t\vert \\
& \quad + e^{C(T_2-T_1)}\int_{\Real} e^{- \vert y(T_1, \xi)-y(T_1, \eta)\vert}(y_\xi+H_\xi)(T_1,\eta) \\
& \qquad \qquad \qquad \qquad \qquad \qquad \qquad \times \mathbbm{1}_{\{\zeta \mid t<\tau(\zeta)\leq s\}}(\eta) d\eta.
\end{align*}

Recalling \eqref{est:P1}, we end up with the following upper bound for $S_1$
\begin{align*}
S_1& = \sum_i \vert p(t_i, y(t_i, \xi_i))-p(t_{i-1}, y(t_{i-1}, \xi_i))\vert\\
& \leq \sum_i \vert P_1(t_i, \xi_i)-P_1(t_{i-1}, \xi_{i})\vert + \sum_{i} \vert P_2(t_i, \xi_i)-P_2(t_{i-1}, \xi_{i})\vert\\
& \leq C e^{ C(T_2-T_1)}\sum_i\vert t_i-t_{i-1}\vert \\
& \quad + e^{C(T_2-T_1)} \sum_i \int_{\Real} e^{-\vert y(T_1, \xi)-y(T_1, \eta)\vert} (y_\xi+H_\xi)(T_1, \eta) \mathbbm{1}_{\{\zeta \mid t_{i-1}<\tau(\zeta)\leq t_i\}}(\eta) d\eta\\
& \leq Ce^{ C(T_2-T_1)} \vert T_2-T_1\vert\\
& \quad + e^{ C(T_2-T_1)} \int_\Real e^{-\vert y(T_1, \xi)-y(T_1, \eta)\vert} (y_\xi+H_\xi) (T_1, \eta) \mathbbm{1}_{\{\zeta \mid T_1<\tau(\zeta)\leq T_2\}}(\eta) d\eta\\
& \leq Ce^{ C(T_2-T_1)} (1+\vert T_2-T_1\vert).
\end{align*}

{\it Estimate for $S_2$:}
Observe that 
\begin{align*}
\vert p(t, y(t,\xi))-p(t,y(s,\xi))\vert& \leq \norm{Q(t,\cdot)}_\infty \vert y(t,\xi)-y(s,\xi)\vert  \leq  C \vert t-s\vert,
\end{align*}
due to \eqref{rel:pPpxQ}, \eqref{Lag:C}, and \eqref{est:yts},
which implies 
\begin{equation*}
S_2= \sum_i \vert p(t_{i-1}, y(t_{i-1}, \xi_i))-p(t_{i-1}, y(t_i, \xi_i))\vert\leq C \sum_i\vert t_i-t_{i-1}\vert \leq C\vert T_2-T_1\vert.
\end{equation*}

{\it Estimate for $S_3$:} Recall that $\sigma(\cdot)$ is Lipschitz continuous with Lipschitz constant $\hat L$,  i.e.,  
\begin{equation*}
\vert \sigma(t)-\sigma(s)\vert \leq \hat L \vert t-s\vert \quad \text{ for all } s,t \in [T_1, T_2]. 
\end{equation*}
Thus following the same lines as for $S_2$ we have 
\begin{align*}
S_3= \sum_{i} \vert p(t_{i-1}, \sigma(t_i))-p(t_{i-1}, \sigma(t_{i-1}))\vert&\leq C \sum_i \vert \sigma(t_i)-\sigma(t_{i-1})\vert \leq C \hat L \vert T_2-T_1\vert.
\end{align*}

\section{Uniqueness of weak dissipative solutions $(u,\nu)$.}\label{sec:unique}

A closer look at \eqref{weak:u} and \eqref{weak:trans} suggests that the most natural approach for computing weak solutions $(u,\nu)$ is to apply the method of characteristics, which yields a unique solution whenever the differential equation for the characteristic is uniquely solvable. In our case, this equation reads
\begin{equation}\label{eq:char}
y_t(t,\xi)= u(t,y(t,\xi)),
\end{equation}
which might have more than one solution, since $u(t,x)$ is only H{\"o}lder continuous with exponent one-half. Nonetheless, we will prove in this section the following result.

\begin{theorem}\label{main}
For any initial data $(u_0, \nu_0)\in \D$ the Camassa--Holm equation has a unique global weak dissipative solution $(u,\nu)$ in the sense of Definition~\ref{def:dissol2}, which coincides with the dissipative solution constructed in \cite{GHR2}.
\end{theorem}

Given a weak, dissipative solution $(u,\nu)$, let
\begin{equation}\label{def:tiy}
\tilde y(t,\zeta)=\sup\{x\mid x+G(t,x)<\zeta\},
\end{equation}
where $G(t,x)=\nu(t,(-\infty,x))$. By construction, $\tilde y(t,\cdot):\mathbb{R}\to \mathbb{R}$ is non-decreasing and Lipschitz continuous with Lipschitz constant at most one, cf. the proof of \cite[Theorem 3.8]{HR2}. Furthermore, introduce
\begin{equation}\label{def:tiH}
\tilde H(t,\zeta)=\zeta-\tilde y(t,\zeta).
\end{equation}
Since $\tilde y(t,\cdot)$ is non-decreasing and Lipschitz continuous with Lipschitz constant at most one, it follows that also $\tilde H(t,\cdot):\mathbb{R}\to \mathbb{R}$ is non-decreasing and Lipschitz continuous with Lipschitz constant at most one. For later use, we define
\begin{align}\label{def:tiU}
\tilde U(t,\zeta)&=u(t,\tilde y(t,\zeta)),\\ 
\tilde V(t,\zeta)&= F(t, \tilde y(t,\zeta)), \\ \label{def:tiP}
\tilde P(t,\zeta)&= p(t, \tilde y(t,\zeta)),\\ \label{def:tiQ}
\tilde Q(t,\zeta)&=p_x(t, \tilde y(t,\zeta)).
\end{align}

\vspace{0.1cm}
{\it Notation:} Let $T>0$ be a number to be specified at the end of this section. Furthermore, we will for the remainder of this section denote by $C$ positive, real constants, which are only dependent on $T$ and $H_\infty(0)$ and which might change from line to line.

\subsection{The integral equation for the characteristics.}

The starting point for deriving \eqref{eq:char} is the study of the inhomogeneous transport equation satisfied by $\nu$, i.e., 
\begin{equation*}
\nu_t+(u\nu)_x= (u^3-2pu)_x.
\end{equation*}
Clever choices of the test function $\phi$ in \eqref{weak:trans}, the weak formulation of the above transport equation, make it possible to prove the following result.

\begin{lemma}\label{lem:fineG}
Let 
\begin{equation}\label{def:G}
G(t,x)=\nu(t,(-\infty,x)),
\end{equation}
then for any $s,t\in [0,T]$ with $t\leq s$
\begin{align*}
G(s, x+  u(t,x)(s-t)&-M(s-t)^{3/2})-N(s-t)^{3/2}\\
& \leq G(t,x)+ \int_t^s (u^3-2pu)(l,  x+ u(t,x)(l-t))dl\\
& \leq G(s, x+u(t,x)(s-t)+M(s-t)^{3/2})+ N(s-t)^{3/2},
\end{align*}
where $M$ and $N$ denote positive constants which are independent of $s$ and $t$.
\end{lemma}

\begin{proof}
We will only prove that there exist positive constants $M$ and $N$ such that 
\begin{align}\label{claim:G2}
G(0, x)\leq G(t, x+ u(0, x)t+ Mt^{3/2})- \int_0^t (u^3-2pu)(l, x+ u(0, x)l) dl+ Nt^{3/2},
\end{align}
since the other inequalities follow using the same argument with slight modifications.

{\it Step 1: Enlarge the set of admissible test functions $\phi$ in \eqref{weak:trans}.}
Note that $\nu(t,\mathbb{R})=D$ for all $t\in \Real$, since $u(t,\cdot)$ and $p(t,\cdot)$ belong to $H^1(\Real)$. In fact, as we will see next, given $\varepsilon>0$, there exists $K>0$ such that 
\begin{equation}\label{claim:step1}
\nu(t,(-K,K))\geq D-\varepsilon \quad \text{ for all }t\in [0,T].
\end{equation}

Let $\delta >0$ be a small number specified later, then there exists $x_\delta>0$, such that 
\begin{equation}\label{decay:u}
\vert u(t,x)\vert <\delta \quad \text{ for all } (t,x)\in [0,T]\times ((-\infty, -x_\delta) \cup (x_\delta, \infty)),
\end{equation}
since $u$ is H{\"o}lder continuous and satisfies $u(t,\cdot)\in H^1(\mathbb{R})$ for all $t\in [0,T]$. Pick $\psi\in C^\infty(\mathbb{R})$ such that 
\begin{equation*}
\psi(-x)=\psi(x), \quad \psi\mid_{[0,x_\delta]}=1, \quad \psi\mid_{[x_\delta+1, \infty)}=0, \quad\text{and}\quad \psi'\mid_{[0,\infty)}\leq 0.
\end{equation*}
Then 
\begin{equation*}
G(t, x_{\delta})-G(t, -x_{\delta}+)\leq \int_{\mathbb{R}} \psi(x)d\nu(t)\leq G(t, x_{\delta}+1)-G(t,- x_{\delta}-1+),
\end{equation*}
which, combined with \eqref{weak:trans}, yields
\begin{align*}
G(t, x_{\delta}+1)-G(t, -x_{\delta}-1+)& \geq \int_{\mathbb{R}}\psi(x)d\nu(t)\\
& =\int_{\mathbb{R}} \psi(x) d\nu(0)+ \int_0^t \int_{\mathbb{R}} u(l,x) \psi'(x) d\nu(l) dl\\
&  \quad -\int_0^t \int_{\mathbb{R}} (u^3-2pu)(l,x)\psi'(x) dx dl\\
& \geq G(0, x_{\delta})- G(0, -x_{\delta} +)\\ 
& \quad + \int_0^t \int_{\mathbb{R}} u(l,x) \psi'(x) d\nu(l) dl\\
&  \quad -\int_0^t \int_{\mathbb{R}} (u^3-2pu)(l,x)\psi'(x) dx dl.
\end{align*}
Since $\supp(\psi')\subset (-1-x_{\delta}, -x_{\delta})\cup (x_{\delta}, x_{\delta}+1)$, \eqref{decay:u}, \eqref{est:uinf}, and \eqref{est:pxpinf} imply
\begin{equation*}
\left\vert \int_0^t \int_{\mathbb{R}} u(l,x) \psi'(x) d\nu(l) dl\right\vert\leq \delta TD \|\psi'\|_\infty
\end{equation*}
and 
\begin{equation*}
\left\vert \int_0^t \int_{\mathbb{R}} (u^3-2pu)(l,x)\psi'(x) dx dl\right\vert \leq \delta T(\delta^2+D)\|\psi'\|_\infty.
\end{equation*}
Thus, choosing first $\delta>0$ such that $\delta T(\delta^2+2D)\|\psi'\|_\infty\leq \frac12 \varepsilon$ and thereafter $K>x_{\delta}+1>0$ such that 
\begin{equation*}
\nu(0,(-(K-1),K-1))=G(0, K-1)-G(0,-(K-1)+)\geq D-\frac12 \varepsilon
\end{equation*}
yields \eqref{claim:step1}.

Let $\psi\in C^\infty(\Real)$ be a monotone increasing function such that $\psi'\in C_c^\infty(\Real)$ and $\psi(x)\to 0$ as $x\to \infty$. Furthermore, denote by $\phi(t,x)\in C^1([0,T]\times \Real)$ the unique solution to $\phi_t+g\phi_x=0$ with initial data $\phi(0,x)=\psi(x)$ for some given continuous function $g(t,x)$ to be specified later. Then, the function $\hat \phi=\phi-\tilde \phi$ belongs to $C^1([0,T]\times \Real)$ and $\hat\phi(t,\cdot)$, $\hat\phi_x(t,\cdot)\in C_0(\Real)$ for all $t\geq 0$, if $\tilde \phi\in C^\infty([0,T]\times\Real)$ such that 
\begin{equation*}
\tilde \phi_x\geq 0
\quad \text{ and } \quad  \tilde \phi(t,x)=\begin{cases} \psi(-\infty) &\quad \text{ for } x\leq -1-K,\\ 0  & \quad \text{ for } -K\leq x,
\end{cases}
\end{equation*}
Moreover it can be shown that \eqref{weak:trans} also holds for $\hat\phi$, since $C_c^\infty(\Real)$ is a dense subset of $C_0(\Real)$, and hence 
\begin{align*}
\int_\Real \phi(t,x) d\nu(t)&= \int_\Real \tilde\phi(t,x)d\nu(t)+ \int_\Real \hat \phi(t,x) d\nu(t)\\
& = \int_\Real \tilde \phi(t,x) d\nu(t)+ \int_\Real \hat \phi(0,x)d\nu(0)\\
& \quad + \int_0^t \int_\Real (\hat \phi_t+ u\hat\phi_x)(l,x) d\nu(l) dl - \int_0^t\int_\Real (u^3-2pu)\hat \phi_x(l,x) dxdl\\
& = \int_\Real \tilde \phi(t,x)d\nu(t)-\int_\Real \tilde \phi(0,x) d\nu(0)\\
& \quad - \int_0^t \int_\Real (\tilde \phi_t+u\tilde \phi_x)(l,x) d\nu(l)dl + \int_0^t \int_\Real(u^3-2pu)\tilde \phi_x(l,x) dx dl\\
& \quad + \int_\Real \phi(0,x) d\nu(0)\\
& \quad + \int_0^t \int_\Real (\phi_t+u\phi_x)(l,x) d\nu(l) dl - \int_0^t \int_\Real (u^3-2pu)\phi_x(l,x) dx dl.
\end{align*}
By \eqref{claim:step1} and \eqref{decay:u}, the terms depending on $\tilde \phi(t,x)$ can be made arbitrarily small by increasing $K$, so that 
\begin{align}\nonumber
\int_\Real \phi(t,x) d\nu(t)& = \int_\Real \phi(0,x) d\nu(0)\\ \label{ext:weaktransp}
& + \int_0^t \int_\Real (u-g)\phi_x(l,x) d\nu(l) dl- \int_0^t \int_\Real (u^3-2pu)\phi_x(l,x) dxdl,
\end{align}
where we also used $\phi_t+g\phi_x=0$.

{\it Step 2: Proof of \eqref{claim:G2}}
Let $\bar x\in \Real$.
Since $u$ is H{\"o}lder continuous with exponent one-half, there exists a constant $\bar D$ such that 
\begin{equation}\label{u:Hol}
\vert u(t,x)-u(s,y)\vert \leq \bar D(\vert t-s\vert+ \vert x-y\vert )^{1/2} \quad \text{ for all } (t,x), (s,y)\in [0,T]\times \Real,
\end{equation}
and, in particular 
\begin{equation*}
u(t,x)\leq u(0,\bar x)+ \bar D(t+\vert x-\bar x\vert )^{1/2} \quad \text{ for all } (t,x) \in [0,T]\times \Real.
\end{equation*}

Choose $g(t,x)$ to be the H{\"o}lder continuous function
\begin{equation*}
g(t,x)= u(0,\bar x)+ \bar D(t+\vert x-\bar x\vert )^{1/2}.
\end{equation*}
Then $g\geq u$ on all of $[0,T]\times \Real$ and $g(0,\bar x)=u(0,\bar x)$. Moreover, let $\psi\in C^\infty(\Real)$ be a monotone increasing function such that $\psi'\in C^\infty_c(\Real)$ and $\psi(x)\to 0$ as $x\to \infty$. Then, it has been shown in the proof of \cite[Lemma 3.3]{GH}, that $\phi(t,x)$, the unique solution to $\phi_t+g\phi_x=0$ with initial data $\phi(0,x)=\psi(x)$, is given implicitly through
\begin{equation*}
\phi(t, \xi(t,z))=\phi(0,\xi(0,z))=\psi(z),
\end{equation*}
where $\xi(t,z)$ is the unique solution to 
\begin{equation*}
\xi_t= g(t,\xi) , \quad \xi(0,z)=z.
\end{equation*}
In addition, the function $\xi(t,\cdot)$ is strictly increasing, continuous, and satisfies for $t\geq 0$ and $z=\bar x$
\begin{equation}\label{est:xibx}
\bar x+ u(0, \bar x)t\leq \xi(t,\bar x)\leq  \bar x+ u(0,\bar x)t+ Mt^{3/2}
\end{equation}
for some positive constant $M$, which does not depend on the particular choice of $\bar x$. Furthermore, 
\begin{align*}
\int_{\Real} \phi(t,x) d\nu(t)& = -\int_\Real \phi_x(t,x) G(t,x) dx\\
& = -\int_\Real \phi_x(t, \xi(t,z)) G(t, \xi(t,z)) \xi_z(t,z)dz\\
& = - \int_\Real \frac{d}{dz} (\phi(t,\xi(t,z))) G(t, \xi(t,z)) dz\\
&=- \int_\Real \frac{d}{dz} (\phi(0, \xi(0,z))) G(t,\xi(t,z)) dz\\
& =- \int_\Real \psi'(z) G(t, \xi(t,z)) dz,
\end{align*}
and 
\begin{align*}
\int_0^t \int_\Real (u^3-2pu)\phi_x(l,x) dxdl=\int_0^t \int_\Real (u^3-2pu)(l, \xi(l, z))\psi'(z) dz.
\end{align*}
In addition, 
$\phi_x(t,x)\geq 0$, which, together with $g \geq u$ on all of $[0,T] \times \Real$, implies 
\begin{equation*}
\int_0^t \int_\Real (u-g)\phi_x(l,x)d\nu(l)dl\leq 0.
\end{equation*}
Thus \eqref{ext:weaktransp} turns into 
 \begin{align*}
 \int_\Real (G(0,z)-G(t, \xi(t,z)))\psi'(z)dz \leq -\int_0^t \int_\Real (u^3-2pu)(l, \xi(l,z))\psi'(z) dzdl.
 \end{align*}
 Since $\psi'$ can be any positive function in $C_c^\infty(\Real)$, we have 
 \begin{equation*}
 G(0,z\pm)\leq G(t, \xi(t,z)\pm)- \int_0^t (u^3-2pu)(l, \xi(l,z)) dl,
 \end{equation*}
 where the integral on the right hand side is well-defined, since $u$ is H{\"o}lder continuous and $p$ is a function of bounded variation along the Lipschitz continuous path $t\mapsto (t, \xi(t,z))$ in the sense of Definition~\ref{def:dissol2} \eqref{cond:dissol2:5}.
Choosing $z=\bar x$ and recalling \eqref{est:xibx}, we obtain
\begin{equation*}
G(0, \bar x)\leq G(t, \bar x+ u(0, \bar x)t+ Mt^{3/2})- \int_0^t (u^3-2pu)(l, \xi(l, \bar x))dl.
\end{equation*}
Using \eqref{est:uinf}, \eqref{est:pxpinf}, \eqref{u:Hol}, \eqref{est:xibx}, and that $p(t,\cdot)$ is Lipschitz continuous once more, it follows that there exists a constant $N>0$, independent of $\bar x$ such that 
\begin{equation*}
-\int_0^t(u^3-2pu)(l, \xi(l, \bar x))dl\leq- \int_0^t (u^3-2pu)(l, \bar x+ u(0, \bar x)l) dl + Nt^{3/2},
\end{equation*}
and hence
\begin{align*}
G(0, \bar x)\leq G(t, \bar x+ u(0, \bar x)t+ Mt^{3/2})- \int_0^t (u^3-2pu)(l, \bar x+ u(0, \bar x)l) dl+ Nt^{3/2}.
\end{align*}
\end{proof}

Given $T>0$, we are now ready to show that Lemma~\ref{lem:fineG} implies that $\tilde y(t,\zeta)$ is Lipschitz continuous on $[0,T]\times\mathbb{R}$ and hence differentiable a.e. on $[0,T]\times \mathbb{R}$.

Let 
\begin{equation}\label{def:I}
I(s,t,x)= \int_t^s (u^3-2pu)(l, x+u(t,x)(l-t))dl,
\end{equation}
which satisfies, cf. \eqref{est:uinf} and \eqref{est:pxpinf}, 
\begin{equation}\label{est:I}
\vert I(s,t,x)\vert \leq C\vert s-t\vert,
\end{equation}
for some $C>0$. 
Furthermore, for any $t\geq 0$, introduce the strictly increasing function $L(t,\cdot):\Real\to \Real$ given by 
\begin{equation}\label{def:L}
L(t,x)=x+G(t,x),
\end{equation}
which satisfies, according to Lemma~\ref{lem:fineG}
\begin{align}\nonumber
L(s, x+u(t,x)(s-t)&-M(s-t)^{3/2}) -N(s-t)^{3/2}\\ \nonumber
&\leq L(t,x)+I(s,t,x) +u(t,x)(s-t)\\ \label{LL:est}
& \leq L(s, x+u(t,x)(s-t)+M(s-t)^{3/2}) + N(s-t)^{3/2}.
\end{align}
In addition, by \eqref{def:tiy},
\begin{equation}\label{prop:L}
L(t,\tilde y(t,\zeta))\leq \zeta \leq L(t, \tilde y(t,\zeta)+), \quad \text{ for all } (t, \zeta)\in \Real^+\times\Real,
\end{equation}
so that choosing $x=\tilde y(t,\zeta)$ in the first inequality in \eqref{LL:est} and using \eqref{def:tiU} yields 
\begin{align*}
L(s, \tilde y(t,\zeta)& +\tilde U(t,\zeta)(s-t)-M(s-t)^{3/2})\\
& \leq \zeta +I(s,t, \tilde y(t,\zeta))+\tilde U(t, \zeta)(s-t)+N(s-t)^{3/2}\\
& \leq L(s, \tilde y(s, \zeta +I(s,t,\tilde y(t,\zeta)) +\tilde U(t,\zeta)(s-t)+N(s-t)^{3/2})+), 
\end{align*}
which implies
\begin{align*}
\tilde y(t,\zeta)& +\tilde U(t,\zeta)(s-t)-M(s-t)^{3/2}\\
& \leq \tilde y(s,\zeta+I(s,t,\tilde y(t,\zeta))+\tilde U(t,\zeta)(s-t)+N(s-t)^{3/2})\\
& \leq \tilde y(s,\zeta+I(s,t,\tilde y(t,\zeta)) +\tilde U(t,\zeta)(s-t))+ N (s-t)^{3/2},
\end{align*}
since $L(t,\cdot):\Real\to \Real$ is strictly increasing and $\tilde y(t,\cdot)$ is Lipschitz continuous with Lipschitz constant at most one.

Following the same lines, but using this time the second inequality in \eqref{LL:est}, yields
\begin{align*}
\tilde y(s,\zeta +I(s,t, \tilde y(t,\zeta))+&\tilde U(t,\zeta)(s-t))- N (s-t)^{3/2}\\
& \qquad \leq \tilde y(t,\zeta) +\tilde U(t,\zeta)(s-t)+M(s-t)^{3/2}.
\end{align*}
Thus we have shown that there exists a constant $C>0$ such that 
\begin{equation}\label{nest:Lip}
\vert \tilde y(t,\zeta)+\tilde U(t,\zeta)(s-t)-\tilde y\big(s,\zeta +I(s,t, \tilde y(t,\zeta))+ \tilde U(t,\zeta)(s-t)\big)\vert \leq C\vert s-t\vert ^{3/2}.
\end{equation}
An immediate consequence of this estimate is that $\tilde y(t,\zeta)$ is Lipschitz continuous on $[0,T]\times \Real$ and hence differentiable a.e. on $[0,T]\times \Real$. Indeed, using \eqref{est:I} and \eqref{nest:Lip} together with $\tilde y(t,\cdot)$ being Lipschitz continuous with Lipschitz constant at most one and $u$ satisfying \eqref{est:uinf} and \eqref{def:tiU}, we have 
that there exists another constant $C>0$, such that 
\begin{align*}
\vert \tilde y(t,\zeta)-\tilde y(s, \eta)\vert & \leq \vert \tilde y(t, \zeta)+\tilde U(t,\zeta)(s-t) -\tilde y\big(s,\zeta +I(s,t, \tilde y(t,\zeta))+\tilde U(t,\zeta)(s-t)\big)\vert \\
& \quad + \vert \tilde y\big(s,\zeta +I(s,t, \tilde y(t,\zeta))+\tilde U(t,\zeta)(s-t)\big)-\tilde y(s, \eta)\vert \\
& \quad + \vert \tilde U(t,\zeta)(s-t)\vert\\
& \leq C(\vert s-t\vert +\vert \zeta-\eta\vert).
\end{align*}

Instead of identifying a differential equation for $\tilde y(t,\zeta)$, the remainder of this section will be devoted to identifying an integral equation of a relabeled version of $\tilde y(t,\zeta)$. To be more precise, 
we show that there exists a mapping $k(t,\xi)$ such that 
\begin{equation}\label{def:yUH}
(y,U, V,H)(t,\xi)= (\tilde y, \tilde U, \tilde V,\tilde H)(t, k(t,\xi))
\end{equation}
satisfies
\begin{equation*}
y(s,\xi)=y(t, \xi)+ \int_t^s U(l, \xi) dl.
\end{equation*}

To identify a suitable function $k(t, \xi)$, we have a closer look at \eqref{nest:Lip} and in particular, at the expression 
\begin{equation}\label{closer:look}
\zeta +I(s,t, \tilde y(t,\zeta))+\tilde U(t,\zeta)(s-t),
\end{equation}
which is an approximation of the solution to the ordinary differential equation
\begin{equation*}
f_l(l,\xi)=(\tilde U^3-2\tilde P\tilde U+ \tilde U)(l, f(l,\xi)), \quad l\geq t
\end{equation*}
with $f(t,\xi)=\xi$. Motivated by this observation, we choose $k(t,\xi)$ to be, if it exists, the solution to 
\begin{equation}\label{Car:ODE}
k_t(t,\xi)=(\tilde U^3-2\tilde P\tilde U+ \tilde U)(t, k(t,\xi)) \quad \text{ with }\quad k(0,\xi)=\xi.
\end{equation}

\begin{lemma}\label{lem:k}
The ordinary differential equation \eqref{Car:ODE}  has a unique solution $k(t,\xi)$ for $t\in [0,T]$. Moreover, the function $k(t,\xi)$ is Lipschitz continuous on $[0,T]\times \Real$. In addition, for any $t\in [0,T]$, 
the function $k(t, \cdot):\Real\to \Real$ is strictly increasing and there exists  a constant $C>0$ such that for any $\xi_1<\xi_2$
\begin{equation}\label{Gron:k}
e^{-Ct} ( \xi_2-\xi_1) \leq k(t,\xi_2)-k(t,\xi_1)\leq e^{Ct} (\xi_2-\xi_1).
\end{equation}
\end{lemma}

\begin{proof}
Definition~\ref{def:dissol2} together with the Lipschitz continuity of $\tilde y(t,\zeta)$ guarantees that the assumptions of Carath{\'e}odory's existence theorem, see e.g. \cite[Chapter 2, Theorem 1.1]{CL}, are satisfied and hence, \eqref{Car:ODE} has for each fixed $\xi\in \Real$ at least one solution, which is absolutely continuous. It therefore remains to show that there exists at most one solution to \eqref{Car:ODE}. Our contradiction argument is based on the fact that for fixed $\xi \in \Real$ solving \eqref{Car:ODE} is equivalent to solving the integral equation
\begin{equation}\label{int:ode}
k(t, \xi)= \xi+\int_0^t (\tilde U^3-2\tilde P\tilde U+ \tilde U)(l, k(l, \xi))dl.
\end{equation}

Given $\xi\in \Real$, assume that there exist two solutions $k_1(t, \xi)$ and $k_2(t, \xi)$ to \eqref{int:ode} with $k_1(0, \xi)=k_2(0, \xi)$. If we can show that the function 
\begin{equation}\label{def:f}
f(t, \zeta)= (\tilde U^3-2\tilde P\tilde U+ \tilde U)(t,\zeta)
\end{equation}
satisfies 
\begin{equation}\label{Lip:f}
\vert f(t,\zeta)-f(t,\eta)\vert \leq C\vert \zeta-\eta\vert \quad \text{ for all } \zeta,\eta\in \Real \text{ and } t\geq 0,
\end{equation}
then we have for fixed $\xi\in \Real$, 
\begin{align*}
\vert k_2-k_1\vert (t, \xi)& \leq \int_0^t \vert f(l, k_2(l, \xi))-f(l, k_1(l, \xi))\vert dl \leq C\int_0^t \vert k_2-k_1\vert (l, \xi) dl,
\end{align*}
and using Gronwall's lemma, we can deduce that $k_2(t,\xi)=k_1(t,\xi)$ for every $t\in [0,T]$.

Since $\tilde U$ and $\tilde P$, given by \eqref{def:tiU} and \eqref{def:tiP}, are uniformly bounded due to \eqref{est:uinf} and \eqref{est:pxpinf}, it suffices to prove \eqref{Lip:f} for $f$ replaced by $\tilde U$ and $\tilde P$. A closer look at  \eqref{def:tiy}, \eqref{def:tiH}, \eqref{def:L}, and \eqref{prop:L} reveals that 
\begin{equation*}
\tilde y(t,\zeta)+G(t, \tilde y(t,\zeta))\leq \tilde y(t,\zeta)+\tilde H(t,\zeta)\leq \tilde y(t,\zeta)+ G(t,\tilde y(t,\zeta)+), 
\end{equation*}
which is equivalent to 
\begin{equation*}
G(t, \tilde y(t,\zeta))\leq \tilde H(t,\zeta)\leq G(t, \tilde y(t,\zeta)+)
\end{equation*}
and hence
\begin{equation*}
\tilde H(t,\zeta)= \sigma G(t, \tilde y(t, \zeta))+ (1-\sigma) G(t, \tilde y(t,\zeta)+) \quad \text{ for some } \sigma\in [0,1].
\end{equation*}
This means especially, given $\zeta \in \Real$ there exist $\zeta^-\leq\zeta \leq\zeta^+$ unique such that 
\begin{equation*}
\tilde y(t, \zeta^-)=\tilde y(t,\zeta)=\tilde y(t, \zeta^+),
\end{equation*}
\begin{equation*}
\tilde H(t, \zeta^-)= G(t, \tilde y(t, \zeta)) \quad \text{ and } \quad \tilde H(t, \zeta^+)=G(t, \tilde y(t, \zeta)+).
\end{equation*}
By the definition of $\D$, we then have for any $\zeta_1<\zeta_2$ such that $\tilde y(t, \zeta_1)\not= \tilde y(t, \zeta_2)$,
\begin{align*}
\vert \tilde U(t, \zeta_2)-\tilde U(t, \zeta_1)\vert 
& = \vert u(t, \tilde y(t, \zeta_2^-))-u(t, \tilde y(t, \zeta_1^+))\vert \\
& \leq \sqrt{\tilde y(t, \zeta_2^-)-\tilde y(t, \zeta_1^+)}\sqrt {\int_{\tilde y(t, \zeta_1^+)}^{\tilde y(t, \zeta_2^-)} u_x^2(t, z) dz}\\
& \leq \sqrt{\tilde y(t, \zeta_2^-)-\tilde y(t, \zeta_1^+)}\sqrt{G(t, \tilde y(t, \zeta_2^-))-G(t,\tilde y(t, \zeta_1^+))}\\
& \leq \sqrt{\tilde y(t, \zeta_2^-)-\tilde y(t, \zeta_1^+)}\sqrt{\tilde H(t, \zeta_2^-)-\tilde H(t, \zeta_1^+)}\\
& \leq \vert \zeta_2^--\zeta_1^+\vert \leq \zeta_2-\zeta_1,
\end{align*}
since both $\tilde y(t,\cdot)$ and $\tilde H(t, \cdot)$ are Lipschitz continuous with Lipschitz constant at most one. 

For $\tilde P(t, \cdot)$ we can use \eqref{def:tiP} combined with \eqref{est:pxpinf} as follows
\begin{align}\label{P:Lip}
\vert \tilde P(t, \zeta_2)- \tilde P(t, \zeta_1)\vert &  \leq \norm{p_x(t,\cdot)}_\infty\vert \tilde y(t, \zeta_2)-\tilde y(t, \zeta_1)\vert  \leq C\vert \zeta_2-\zeta_1\vert,
\end{align}
where $C$ denotes a positive constant. This finishes the proof of \eqref{Lip:f} and hence \eqref{int:ode} has exactly one solution. 

Next, note that for any $\xi$ and $\eta \in \Real$, by \eqref{Lip:f},
\begin{equation}\label{lb:kxi2}
\vert k(t, \xi)-k(t, \eta)\vert \leq \vert \xi-\eta\vert + C\int_0^t \vert k(l, \xi)-k(l, \eta)\vert dl,
\end{equation}
and, by Gronwall's inequality,
\begin{equation}\label{k:Lip1}
 \vert k(t,\xi)-k(t,\eta)\vert  \leq e^{Ct}\vert \xi-\eta\vert \text{ for any } t\in [0,T].  
\end{equation}
Furthermore, $\tilde P$ and $\tilde U$ are uniformly bounded due to \eqref{est:uinf} and \eqref{est:pxpinf}, and hence there exists a constant $C$ such that 
\begin{equation}\label{k:Lip2}
\vert k(t,\xi)-k(s, \xi)\vert \leq C\vert t-s\vert \quad \text{ for all } s, t\in [0,T] \text{ and }\zeta \in \Real.
\end{equation}
Thus \eqref{k:Lip1} and \eqref{k:Lip2} imply that $k$ is Lipschitz continuous and hence differentiable almost everywhere on $[0,T]\times \Real$.

We show by contradiction that for each $t\in [0,T]$, the function $k(t, \cdot)$ is strictly increasing. Assume the opposite, i.e., there exists $\hat t\in [0,T]$ and $\xi_1<\xi_2$ such that  
\begin{equation*}
k(t, \xi_1)< k(t, \xi_2) \quad \text{ for all } t\in [0, \hat t) \quad \text{ and } \quad k(\hat t, \xi_1)=k(\hat t, \xi_2).
\end{equation*}
By \eqref{Lip:f},
\begin{equation*}
-C\int_s^{\hat t} k(l, \xi_2)-k(l, \xi_1)dl + k(s, \xi_2)-k(s, \xi_1)\leq k(\hat t, \xi_2)-k(\hat t, \xi_1),
\end{equation*}
for all $0\leq s\leq \hat t$ and defining $\tilde k(l,\xi_2)=k(\hat t-l,\xi_2)$, it follows that 
\begin{equation*}
\tilde k(s, \xi_2)-\tilde k(s, \xi_1)\leq \tilde k(0, \xi_2)-\tilde k(0, \xi_1)+ C\int_0^s \tilde k(l, \xi_2)-\tilde k(l, \xi_1),
\end{equation*}
for all $0\leq s\leq \hat t$. Applying Gronwalls inequality then yields 
\begin{align}\label{lb:kxi}
\xi_2-\xi_1& = \tilde k(\hat t, \xi_2)-\tilde k(\hat t, \xi_1)\\ \nonumber
& \leq e^{C\hat t} (\tilde k(0, \xi_2)-\tilde k(0,\xi_1)) = e^{C\hat t} (k(\hat t, \xi_2)-k(\hat t, \xi_1))=0, 
\end{align}
which contradicts the assumption $\xi_1<\xi_2$. Thus, for each $t\in [0,T]$, $k(t, \cdot)$ is a strictly increasing function and \eqref{Gron:k} follows from \eqref{k:Lip1} and \eqref{lb:kxi}.
\end{proof}

Thus, we have proven that $k(t,\xi)$ defines a change of variables on $[0,T] \times \Real$ through the mapping
\begin{equation*}
(t, \xi) \to (t, \zeta)=(t,k(t,\xi)),
\end{equation*}
which means that $y$, $U$, $V$, and $H$ given by \eqref{def:yUH} are well-defined. Furthermore, we have, by \eqref{int:ode} and \eqref{def:I}, for $s\geq t$,
\begin{align*}
k(t,\xi)& + I(s,t, y(t,\xi))+ U(t,\xi)(s-t)\\
&=k(s, \xi)- \int_t^s (\tilde U^3-2\tilde P\tilde U+ \tilde U)(l, k(l, \xi)) dl\\
& \quad + \int_t^s (u^3-2pu)(l, y(t,\xi)+U(t,\xi)(l-t))dl + U(t,\xi)(s-t)\\
& = k(s, \xi) - \int_t^s u(l, y(l, \xi))-u(t, y(t,\xi))dl\\
& \quad -\int_t^s (u^3-2pu)(l, y(l, \xi))-(u^3-2pu)(l, y(t,\xi)+ U(t,\xi)(l-t))dl.
\end{align*}
Recalling \eqref{est:uinf}, \eqref{est:pxpinf}, and \eqref{u:Hol} and observing that $y(t, \xi)$ is Lipschitz continuous, since $\tilde y(t,\zeta)$ and $k(t,\xi)$ are, we have 
\begin{equation}\label{err:k}
\vert k(t,\xi)-k(s, \xi) + I(s,t, y(t,\xi))+ U(t,\xi)(s-t)\vert  \leq C ( s-t)^{3/2},
\end{equation}
where  $C$ denotes a positive constant. Keeping \eqref{err:k} in mind and recalling \eqref{nest:Lip} with $\zeta=k(t,\xi)$ we end up with 
\begin{align}\nonumber
\vert y(s,\xi)&-y(t, \xi)  -\int_t^s U(l, \xi) dl\vert \\ \nonumber
& \leq\vert y(s, \xi)-y(t, \xi) - U(t, \xi) (s-t)\vert + \vert \int_t^s U(l, \xi)-U(t, \xi) dl\vert \\ \nonumber
& \leq \vert \tilde y(s, k(s, \xi))- \tilde y(s, k(t,\xi)+ I(s,t,y(t, \xi))+ U(t, \xi) (s-t))\vert \\ \nonumber
& \quad + \vert \tilde y(s, k(t,\xi)+ I(s,t,y(t, \xi))+ U(t, \xi) (s-t))- y(t,\xi)-U(t,\xi)(s-t)\vert \\ \nonumber
& \quad + \vert \int_t^s u(l, y(l, \xi))-u(t,y(t,\xi)) dl\vert \\ \label{last:est}
& \leq C (s-t)^{3/2}.
\end{align}

Thus for any $\Delta t>0$ such that $N\Delta t= s-t$, we have, using \eqref{last:est},
\begin{align*}
\vert y(s,\xi)-y(t,\xi)& -\int_t^s U(l, \xi) dl\vert \\
& \leq \sum_{n=1}^N \vert y(t+ n \Delta t, \xi)-y(t+ (n-1) \Delta t, \xi)- \int_{t+(n-1)\Delta t}^{t+n\Delta t} U(l, \xi) dl\vert \\
& \leq C \sum_{n=1}^N \Delta t^{3/2}= C T \Delta t^{1/2}.
\end{align*}
Since we can choose any $\Delta t>0$, it follows that 
\begin{equation*}
\vert y(s,\xi)-y(t,\xi)-\int_t^s U(l, \xi) dl\vert=0,
\end{equation*}
which means that $y(t,\xi)$ satisfies for any $\xi\in \Real$ the integral equation 
\begin{equation}\label{inteq:y}
y(s, \xi)= y(t,\xi)+ \int_t^s U(l, \xi) dl=  y(t, \xi)+ \int_t^s u(l, y(l, \xi)) dl.
\end{equation}

Here some observations are important. First, $y(t, \xi)$ is Lipschitz continuous and $U(t,\xi)$ is H{\"o}lder continuous with exponent one-half on $[0,T]\times \Real$. Thus we can write
\begin{equation*}
y_t(t,\xi)= U(t,\xi)=u(t,y(t,\xi)),
\end{equation*}
and in particular, $y(t,\xi)$ is a characteristic. 

Second, introduce
\begin{equation} \label{def:P}
P(t,\xi)= \tilde P(t, k(t,\xi)),
\end{equation}
so that $k(t,\xi)$, given by \eqref{int:ode} reads
\begin{equation*}
k(t,\xi)= \xi + \int_0^t (U^3-2PU+U)(l, \xi) dl.
\end{equation*}
Then we can write, 
combining \eqref{def:tiH} and \eqref{def:yUH}, 
\begin{equation*}
y(t,\xi)+H(t,\xi)= k(t,\xi),
\end{equation*}
which yields, using \eqref{inteq:y} the integral equation,
\begin{equation}\label{inteq:H}
H(s, \xi) = H(t, \xi)+ \int_t^s (U^3-2PU) (l, \xi)dl.
\end{equation}

\subsection{The integral equation for $U$.} Introduce 
\begin{equation}\label{def:Q}
Q(t,\xi)=\tilde Q(t,k(t,\xi)),
\end{equation}
where $\tilde Q(t,\zeta)$ and $k(t,\xi)$ are given by \eqref{def:tiQ} and \eqref{Car:ODE}, respectively.
The goal of this section is to show that 
\begin{equation}\label{inteq:U}
U(s, \xi)= U(t,\xi)-\int_t^s Q(l, \xi) dl.
\end{equation}

According to Definition~\ref{def:dissol2} \eqref{cond:dissol2:7}, one has for all $\phi\in C^\infty_c([t,s]\times \Real)$ with $t<s$  
\begin{equation*}
\int_t^s \int_\Real (u\phi_t+ \frac12 u^2\phi_x-p_x\phi)(l,x)dx dl= \int_\Real u\phi(s,x) dx-\int_\Real u\phi(t,x) dx,
\end{equation*}
but also here the set of admissible test functions can be enlarged. For example, the above equality remains valid for \begin{equation*}
\phi_\varepsilon(t,x)= \frac1{\varepsilon}\psi\left(\frac{y(t,\xi)-x}{\varepsilon}\right),
\end{equation*}
where $\varepsilon>0$ and $\psi$ is a standard Friedrichs mollifier and hence belongs to $C^\infty_c(\Real)$. Furthermore,
\begin{equation*}
\lim_{\varepsilon\to 0} \int_\Real u\phi_\varepsilon(t,x)dx= u(t, y(t,\xi))= U(t,\xi),
\end{equation*}
and hence 
\begin{align}\nonumber
U(s,\xi)-U(t,\xi)& =\lim_{\varepsilon\to 0} \int_\Real (u\phi_\varepsilon(s,x)-u\phi_\varepsilon(t,x)) dx\\ \label{dif:U}
& = \lim_{\varepsilon\to 0} \int_t^s \int_\Real \big(u\phi_{\varepsilon,t}+ \frac12 u^2\phi_{\varepsilon,x}-p_x\phi_\varepsilon\big)(l,x) dxdl,
\end{align}
where the above limit can be evaluated using the dominated convergence theorem. 

Indeed, given $\varepsilon>0$, introduce the function $I_{\varepsilon}:[0,T]\to \Real$ given by 
\begin{align*}
I_\varepsilon(l)&=\int_\Real (u\phi_{\varepsilon,t}+\frac12 u^2\phi_{\varepsilon,x}-p_x\phi_\varepsilon)(l,x) dx\\
& = \int_\Real (u\phi_{\varepsilon,t}+\frac12 u^2\phi_{\varepsilon,x})(l,x)dx+\int_\Real p\phi_{\varepsilon,x}(l,x) dx\\
& = I_{\varepsilon,1}(l)+ I_{\varepsilon,2}(l).
\end{align*}
Since $\psi$ is a standard Friedrichs mollifier, \eqref{inteq:y} and \eqref{est:uinf} imply that there exists a constant $M$, independent of $\varepsilon$, such that 
\begin{equation*}
\supp\phi_\varepsilon(t,\cdot)\in [-M,M] \quad \text{ for all }t\in [0,T].
\end{equation*}
Recalling Definition~\ref{def:dissol2} \eqref{cond:dissol2:2}, \eqref{est:uinf}, and \eqref{inteq:y}, we can therefore conclude that $I_{\varepsilon,1}(t)$ is continuous and hence measurable. 

For $I_{\varepsilon,2}(t)$, on the other hand, we will show that it is a function of bounded variation on $[0,T]$ and hence measurable. Using Definition~\ref{def:dissol2} \eqref{cond:dissol2:5}, \eqref{est:pxpinf}, \eqref{est:uinf}, and \eqref{inteq:y}, we have for any finite partition $\{t_i\}$ of $[0,T]$ with $t_i<t_{i+1}$, 
\begin{align*}
\sum_{i} \vert I_{\varepsilon,2}(t_i)-I_{\varepsilon,2}(t_{i-1})\vert & \leq \sum_i \vert \int_{-M}^M (p(t_i,x)-p(t_{i-1},x))\phi_{\varepsilon,x}(t_i,x)dx \vert \\
& \quad + \sum_i \vert \int_{-M}^M p(t_{i-1},x)(\phi_{\varepsilon,x}(t_i,x)-\phi_{\varepsilon,x}(t_{i-1},x))dx \vert \\
& \leq \frac{2M}{\varepsilon^2} \norm{\psi'}_\infty \sup_{x\in[-M,M]} T.V.(p(\cdot,x))\\
& \quad + \frac{2M C}{\varepsilon^3}\norm{\psi''}_\infty \sum_i \vert y(t_i, \xi)-y(t_{i-1}, \xi)\vert\\
& \leq C (1+ T).
\end{align*}
Thus, $I_{\varepsilon}$ is a measurable function on [0,T].

Next we compute the limit of $I_\varepsilon(l)$ as $\varepsilon\to 0$. Write
\begin{align*}
I_\varepsilon(l)&= \int_\Real  (u\phi_{\varepsilon,t}+ \frac12 u^2\phi_{\varepsilon,x}-p_x\phi_\varepsilon)(l,x) dx\\
& \qquad \qquad = \frac12 u(l, y(l, \xi))^2 \int_\Real \frac1{\varepsilon^2}\psi'\left(\frac{y(l,\xi)-x}{\varepsilon}\right)dx\\
& \quad \qquad \qquad - \frac12 \int_\Real (u(l,x)-u(l,y(l, \xi)))^2\frac1{\varepsilon^2} \psi'\left(\frac{y(l,\xi)-x}{\varepsilon}\right)dx\\
&\quad  \qquad \qquad -\int_\Real (p_x(l,x)-p_x(l, y(l, \xi)))\frac{1}{\varepsilon} \psi\left(\frac{y(l,\xi)-x}{\varepsilon}\right) dx\\
& \quad \qquad \qquad - p_x(l, y(l, \xi))\\
& \qquad \qquad = - \frac12 \int_\Real (u(l,x)-u(l,y(l, \xi)))^2\frac1{\varepsilon^2} \psi'\left(\frac{y(l,\xi)-x}{\varepsilon}\right)dx\\
& \quad \qquad \qquad -\int_\Real (p_x(l,x)-p_x(l, y(l, \xi)))\frac{1}{\varepsilon} \psi\left(\frac{y(l,\xi)-x}{\varepsilon}\right) dx\\
& \quad \qquad \qquad - p_x(l, y(l, \xi)).
\end{align*}

For the first integral on the right hand side observe that $F(l,\cdot)$, given by \eqref{def:F}, is absolutely continuous and satisfies
\begin{equation*}
\vert u(l,x)-u(l,y)\vert\leq \vert x-y\vert^{1/2}\vert F(l,x)-F(l,y)\vert^{1/2},
\end{equation*}
which implies
\begin{align*}
& \left\vert \int_\Real (u(l,x)-u(l,y(l, \xi)))^2\frac1{\varepsilon^2} \psi'\left(\frac{y(l,\xi)-x}{\varepsilon}\right)dx\right\vert\\
& \qquad \qquad \qquad = \frac{1}{\varepsilon} \left\vert \int_{-1}^1 (u(l, y(l, \xi)-\varepsilon \eta)-u(l, y(l, \xi)))^2 \psi'(\eta) d\eta\right\vert \\
&\qquad \qquad \qquad  \leq \vert F(l, y(l, \xi)+\varepsilon)-F(l, y(l, \xi)-\varepsilon)\vert. 
\end{align*}
Since $F(l, \cdot)$ is absolutely continuous, 
\begin{equation*}
\lim_{\varepsilon\to 0}\left\vert \int_\Real (u(l,x)-u(l,y(l, \xi)))^2\frac1{\varepsilon^2} \psi'\left(\frac{y(l,\xi)-x}{\varepsilon}\right)dx\right\vert=0.
\end{equation*}

 For the second integral term on the right hand side, we can write 
 \begin{align*}
 & \left\vert \int_\Real (p_x(l, x)-p_x(l, y(t,\xi)))\frac1{\varepsilon} \psi\left(\frac{y(t,\xi)-x}{\varepsilon}\right)dx \right\vert\\
 & \qquad \qquad \qquad = \left\vert \int_{-1}^1 (p_x(l, y(l, \xi)-\varepsilon\eta)-p_x(l, y(l, \xi)))\psi(\eta) d\eta\right\vert\\
 & \qquad \qquad \qquad \leq \sup_{\eta\in [-1,1]}\vert p_x(l, y(l, \xi)+\varepsilon\eta)-p_x(l, y(l, \xi))\vert ,
 \end{align*}
 and, by  \eqref{cond:limpx}, 
 \begin{equation*}
 \lim_{\varepsilon\to 0}  \left\vert \int_\Real (p_x(l, x)-p_x(l, y(t,\xi)))\frac1{\varepsilon} \psi\left(\frac{y(t,\xi)-x}{\varepsilon}\right)dx \right\vert=0.
 \end{equation*}
 Thus, recalling  \eqref{def:Q},
 \begin{equation*}
 \lim_{\varepsilon\to 0 } I_\varepsilon(l)= -p_x(l, y(l, \xi))=-Q(l, \xi).
 \end{equation*}
 
 A closer look at the above estimates reveals that we have shown in addition, that there exists a constant $\tilde M$ such that 
 \begin{equation*}
\vert  I_\varepsilon(l)\vert \leq \tilde M \quad \text{ for all } \varepsilon>0 \text{ and } l\in [0,T].
\end{equation*}
Since $T$ is finite, $\tilde M$ can be seen as a dominating function and hence all assumptions of the dominated convergence theorem are satisfied. Thus, by \eqref{dif:U},
 \begin{equation*}
 U(s,\xi)-U(t,\xi)=\lim_{\varepsilon\to 0} \int_t^s I_{\varepsilon}(l)dl =-\int_t^s Q(l, \xi) dl,
 \end{equation*}
 and the function $Q(\cdot, \xi)$ is integrable. 
 Furthermore, we have due to \eqref{est:pxpinf} 
 \begin{equation*}
 \vert U(t,\xi)-U(s,\xi)\vert \leq C\vert t-s\vert, 
 \end{equation*}
 where $C$ is a positive constant and recalling that $\tilde U(t,\cdot)$ and $k(t, \cdot)$ are Lipschitz continuous, it follows that $U(t,\xi)$ is Lipschitz continuous on $[0,T]\times \Real$.
 
 \subsection{The integral equation for $y_\xi$}
 So far we have shown that $y$, $U$, and $H$ are solutions to the following system of integral equations
 \begin{subequations}\label{sys:intyUH}
 \begin{align}\label{sys:intyUH1}
 y(t,\xi)&=y(0, \xi)+\int_0^t  U(l,\xi)dl,\\ \label{sys:intyUH2}
 U(t,\xi)&=U(0, \xi)-\int_0^t Q(l,\xi)dl,\\ \label{sys:intyUH3}
 H(t,\xi)&=H(0, \xi)+ \int_0^t  (U^3-2PU)(l,\xi)dl.
 \end{align}
 \end{subequations}
In addition, we have 
 \begin{equation}\label{ident:relcl}
 y(t,\xi)+H(t,\xi)= k(t,\xi),
 \end{equation}
 where $k(t,\xi)$ is the unique solution to \eqref{Car:ODE}. As we will see next $k(t,\xi)$ does not only define a change of variables on $[0,T]\times \Real$. In fact, for each $t\geq 0$ the function $k(t,\cdot)$ is a relabelling function. Recalling Lemma~\ref{rellem} and Lemma~\ref{lem:k} the claim follows if $k_\xi(t,\cdot)-1 $ belongs to $L^2(\Real)$ for all $t\in [0,T]$.
 
 For $x= (x_1, x_2, x_3)\in \mathbb{R}^3$, denote by $\norm{\cdot}$ the following norm on $\Real^3$,
 \begin{equation*}
 \norm{x}=\vert x_1\vert + \vert x_2\vert + \vert x_3\vert.
 \end{equation*}
  Then, the vector $Z=(y-\id,U, H)$ satisfies the following lemma.

 \begin{lemma}
 Given $Z=(y-\id,U,H)$, $\xi<\bar \xi$, and $0\leq t<s\leq T$, then 
 \begin{align} \nonumber
 \|(Z(s,\bar \xi)&-Z(s,\xi))-(Z(t,\bar \xi)-Z(t,\xi))\|\\ \nonumber
& \leq (e^{C(s-t)}-1)\norm{Z(t,\bar \xi)-Z(t,\xi)}\\ \label{help:est}
& \phantom{a}+ Ce^{C(s-t)}\int_t^s \vert U(l,  \xi)\vert + \max_{\hat \xi\in [\xi, \bar\xi]} U^2(l,\hat \xi) +\hat P(l, \xi)+\hat P(l, \bar\xi)dl \vert \bar\xi-\xi\vert,
\end{align}
where 
 \begin{equation}\label{def:hatP}
 \hat P(t, \xi)= \frac14 \int_\Real e^{-\vert y(t,\xi)-y(t,\eta)\vert} (U^2y_\xi+H_\xi)(t,\eta) d\eta.
 \end{equation} 
 \end{lemma}
 
 \begin{proof}
 Writing 
\begin{align*}
Q(t,\bar \xi)-Q(t,\xi)& =- \frac14 \int_{-\infty}^{\xi} (e^{-y(t, \bar \xi)}-e^{-y(t,\xi)})e^{y(t,\eta)} (U^2y_\xi+V_\xi)(t,\eta) d\eta\\
& \quad +\frac14 \int_{\bar \xi}^\infty (e^{y(t,\bar \xi)}-e^{y(t,\xi)})e^{-y(t,\eta)}(U^2y_\xi+V_\xi)(t,\eta) d\eta\\
& \quad - \frac14 \int_{\xi}^{\bar \xi} (e^{y(t,\eta)-y(t,\bar \xi)}+e^{y(t,\xi)-y(t,\eta)})(U^2y_\xi+V_\xi)(t,\eta)d \eta, 
\end{align*}
and applying $e^{b}-e^{a}\leq e^{b}(b-a) $ for $a<b$, yields
\begin{align*}
\vert Q(t,\bar \xi)-Q(t,\xi)\vert &\leq  (P(t, \xi)+P(t,\bar \xi))\vert y(t,\bar\xi)-y(t,\xi)\vert \\
& \quad + \frac12 \left(\max_{\hat\xi\in [\xi, \bar \xi] }U^2(t,\hat \xi)\vert y(t,\bar\xi)-y(t,\xi)\vert+ \vert H(t,\bar\xi)-H(t,\xi)\vert\right).
\end{align*}
Furthermore, using \eqref{def:P} \eqref{P:Lip}, and \eqref{def:yUH}, there exists a constant $C>0$ such that 
\begin{align*}
\vert (U^3-2PU)(t,\bar\xi)& -(U^3-2PU)(t,\xi)\vert \\
&  \leq  C\vert U(t,\xi)\vert  \vert y(t,\bar\xi)-y(t,\xi)\vert \\
& + (2U^2(t,\xi)+2U^2(t,\bar\xi)+P(t,\bar\xi))\vert U(t,\bar\xi)-U(t,\xi)\vert\\
& \leq C(\vert U(t,\xi)\vert \vert y(t,\bar\xi)-y(t,\xi)\vert+ \vert U(t,\bar\xi)-U(t,\xi)\vert)
\end{align*}
and hence, by \eqref{sys:intyUH},
\begin{align*}
& \norm{ (Z(s,\bar\xi)-Z(s,\xi))-(Z(t,\bar\xi)-Z(t,\xi))}\\
 & \qquad \qquad \qquad  \leq \int_t^s \vert U(l,\bar\xi)-U(l,\xi)\vert + \vert Q(l,\bar\xi)-Q(l,\xi)\vert \\
& \qquad \qquad \qquad \qquad \qquad \quad  + \vert (U^3-2PU)(l,\bar\xi)-(U^3-2PU)(l, \xi)\vert dl\\
& \qquad \qquad \qquad  \leq C\int_t^s \norm {Z(l, \bar\xi)-Z(l, \xi)} dl \\
& \qquad \qquad \qquad \quad+ C\int_t^s \vert U(l, \xi)\vert + \max_{\hat \xi\in [\xi, \bar\xi]} U^2(l,\hat\xi) +P(l, \xi)+P(l, \bar\xi)dl \vert \bar\xi-\xi\vert\\
& \qquad \qquad \qquad \leq C\norm{Z(t,\bar\xi)-Z(t,\xi)}(s-t)\\
& \qquad \qquad \qquad \quad  + C\int_t^s \norm{(Z(l,\bar\xi)-Z(l,\xi))-(Z(t,\bar\xi)-Z(t,\xi))}dl\\
& \qquad \qquad \qquad \quad+ C\int_t^s \vert U(l, \xi)\vert + \max_{\hat \xi\in [\xi, \bar\xi]} U^2(l,\hat \xi) +\hat P(l, \xi)+\hat P(l,\bar \xi)dl \vert \bar\xi-\xi\vert,
\end{align*}
since 
\begin{equation*}
\vert y(t,\bar\xi)-y(t,\xi)\vert \leq \vert (y(t,\bar\xi)-\bar\xi)-(y(t,\xi)-\xi)\vert + \vert \bar\xi-\xi\vert.
\end{equation*}
Applying to the above inequality a Gronwall type argument yields \eqref{help:est}.
 \end{proof}
 
 In fact we have shown that each component of $Z(s,\cdot)- Z(t, \cdot)$ is Lipschitz continuous, and hence we obtain as an immediate consequence the following corollary.
 
 \begin{corollary}\label{l2:k}
 Given $Z=(y-\id, U, H)$ and $0\leq t<s\leq T$, then for almost every $\xi\in \Real$, 
  \begin{align*}
\norm{Z_\xi(s,\xi)-Z_\xi(t,\xi)}&\leq (e^{C(s-t)}-1)\norm{Z_\xi(t,\xi)}\\
& \quad + Ce^{C(s-t)}\int_t^s \vert U(l, \xi)\vert +U^2(l, \xi)+2\hat P(l, \xi) dl,
\end{align*}
where $\hat P$ is given by \eqref{def:hatP}.
 \end{corollary}
 
Since the above estimate holds pointwise almost everywhere, we obtain the following norm estimates as a  consequence.

\begin{lemma}\label{lem:finest}
For $p=2$ and $p=\infty$, $0\leq t<s\leq T$, and $f$ equal to $y_\xi-1$, $U_\xi$, or $H_\xi$, there exists a constant $C>0$ such that 
\begin{equation*}
\norm{f(s,\cdot)-f(t, \cdot)}_p \leq Ce^{C(s-t)}(\norm{y_\xi(t,\cdot)-1}_p+\norm{U_\xi(t,\cdot)}_p+ \norm{H_\xi(t,\cdot)}_p+1)(s-t)
\end{equation*}
and 
\begin{equation*}
\norm{f(t,\cdot)}_p \leq C(\norm{y_\xi(0,\cdot)-1}_p+\norm{U_\xi(0,\cdot)}_p+ \norm{H_\xi(0,\cdot)}_p+t).
\end{equation*}
\end{lemma}

\begin{proof}
Without loss of generality we assume that $f=y_\xi-1$ and $t=0$.

Observe that by \eqref{sys:intyUH3} 
\begin{equation}\label{H:preser}
\lim_{\xi\to \infty} H(t,\xi)=\lim_{\xi\to \infty} H(0, \xi)=\nu(0,\Real),
\end{equation}
and hence, by \eqref{ident:relcl} and \eqref{int:ode},
\begin{align*}
0\leq \hat P(t,\xi)& \leq \frac{1}{2} \int_\Real e^{-\vert y(t,\xi)-y(t,\eta)\vert} H_\xi(t,\eta)d\eta\\
& \leq  e^{2\norm{y(t,\cdot)-\id}_\infty}\int_\Real e^{-\vert \xi-\eta\vert} H_\xi(t,\eta) d\eta\\
& \leq e^{2(\norm{k(t,\cdot)-\id} + H_\infty)} \int_\Real e^{-\vert \xi-\eta\vert }H_\xi(t, \eta) d\eta\\
& \leq C\int_\Real e^{-\vert \xi-\eta\vert }H_\xi(t, \eta) d\eta,
\end{align*}
which implies that there exists a constant $C>0$ such that 
\begin{equation}\label{est:Pp}
\norm{\hat P(t, \cdot)}_p\leq C, \quad \text{ for all } t\in [0,T] \text{ and } p=2, \infty,
\end{equation}
due to Young's inequality.
For $U$, \eqref{est:uinf}, \eqref{ident:relcl}, \eqref{Gron:k} and \eqref{H:preser} imply
\begin{equation*}
\norm{U(t,\cdot)}_2^2\leq e^{Ct} (\norm{U^2y_\xi(t,\cdot)}_1+ \norm{U^2H_\xi(t,\cdot)}_1) \leq e^{Ct} (1+ \norm{u(t,\cdot)}_\infty^2) H_\infty\leq C,
\end{equation*}
so that 
\begin{equation}\label{est:Up}
\norm{U(t,\cdot)}_p \leq C, \quad \text{ for all }t\in [0,T] \text{ and } p=2, \infty.
\end{equation}

Keeping these estimates in mind, Corollary~\ref{l2:k} implies 
\begin{align*}
\norm{y_\xi(s, \cdot)-1}_p&\leq 3e^{Cs}(\norm{y_\xi(0,\cdot)-1}_p+\norm{U_\xi(0,\cdot)}_p+ \norm{H_\xi(0,\cdot)}_p)\\
& \qquad +Ce^{Cs}\int_0^s \norm{U(l, \cdot)}_p+ \norm{U^2(l, \cdot)}_p + 2\norm{\hat P(l, \cdot)}_p  dl\\
&\leq Ce^{Cs}(\norm{y_\xi(0,\cdot)-1}_p+\norm{U_\xi(0,\cdot)}_p+ \norm{H_\xi(0,\cdot)}_p+s),
\end{align*}
and $y_\xi(s, \cdot)-1\in L^2(\Real)$.
The other inequality follows now immediately from Corollary~\ref{l2:k}.
 \end{proof}
 
 Turning our attention back to $k(t,\xi)$, we have that for $t=0$, $k(0, \xi)=\xi$ and $k_\xi(0,\cdot)-1$, $y_\xi(0,\cdot)-1$, $U_\xi(0,\cdot)$ and $H_\xi(0,\cdot)$ belong to $L^2(\Real)\cap L^\infty(\Real)$. For $t\in [0,T]$, it now follows, from Lemma~\ref{lem:finest} that $y_{\xi}(t,\cdot)-1$, $U_\xi(t,\cdot)$, and $H_\xi(t,\cdot)$ belong to $L^2(\Real)\cap L^\infty(\Real)$ and especially
\begin{equation*}
k_\xi(t,\cdot)-1= y_\xi(t,\cdot)+H_\xi(t,\cdot)-1\in L^2(\Real)\cap L^\infty(\Real) \quad \text{ for all } t\in[0,T].
\end{equation*}
Thus for every $t\in [0,T]$, $k(t,\cdot)$ is a relabeling function and hence $(y,U,V,H)(t, \cdot)$ belongs to $\F$.\

Finally, we can turn our attention towards the integral equation for $y_\xi(t,\xi)$. Since $y(t,\xi)$ is a solution to \eqref{sys:intyUH1}, we can write for any $t<s$ and any $\xi<\bar \xi$
\begin{align*}
y(s,\bar\xi)-y(s, \xi)& = y(t, \bar\xi)-y(t, \xi)+ (U(t, \bar\xi)-U(t, \xi))(s-t) \\
& \qquad \quad + \int_t^s (U(l, \bar\xi)-U(l,\xi))-(U(t, \bar\xi)-U(t,\xi)) dl,
\end{align*}
which yields the following inequality, when applying \eqref{help:est}, dividing the whole inequality by $\bar\xi-\xi$ and taking the limit as $\bar\xi\to\xi$, 
\begin{align*}
\vert y_\xi(s, \xi)& - y_\xi(t, \xi)- U_\xi(t, \xi)(s-t)\vert  \\
&\qquad  \leq  e^{C(s-t)} (s-t)^2 \|Z_\xi(t, \xi)\| \\
& \qquad \quad + Ce^{C(s-t)}(s-t) \int_t^s \vert U(l, \xi)\vert + U^2(l,\xi)+ 2\hat P(l, \xi)dl,
\end{align*}
which is valid for almost every $\xi\in \Real$.
We therefore have, using Minkowski's inequality for integrals, for $p=2$ and $p=\infty$, Lemma~\ref{lem:finest}, \eqref{est:Pp}, and \eqref{est:Up},
\begin{align}\label{finest:y}
\norm{y_\xi(s, \cdot)-y_\xi(t, \cdot)- U_\xi(t, \cdot)(s-t)}_{p}  \leq C(s-t)^2.
\end{align}
Using once more Minkowski's inequality for integrals, we can also conclude that for any $t<s$,
\begin{equation*}
y_\xi(s,\cdot)-y_\xi(t,\cdot)- \int_t^sU_\xi(l, \cdot) dl \in L^2(\Real)\cap L^\infty(\Real).
\end{equation*}
Thus for any $\Delta t$ such that $N\Delta t=s-t$, we have, by Lemma~\ref{lem:finest} and \eqref{finest:y},
\begin{align*}
& \norm{y_\xi(s,\cdot) -y_\xi(t,\cdot)- \int_t^s U_\xi(l, \cdot) dl}_p\\ &\qquad\qquad   \leq \sum_{n=1}^N \norm{y_\xi(t+n\Delta t,\cdot)-y_\xi(t+(n-1)\Delta t,\cdot)- \int_{(n-1)\Delta t}^{n\Delta t} U_\xi(t+l, \cdot) dl}_p\\
&\qquad \qquad  \leq \sum_{n=1}^N \norm{y_\xi(t+n\Delta t,\cdot)-y_\xi(t+(n-1)\Delta t,\cdot)- \Delta t U_\xi(t+ (n-1)\Delta t, \cdot)}_p\\
& \qquad \qquad \quad + \sum_{n=1}^N \norm{\int_{(n-1)\Delta t}^{n\Delta t} U_\xi(t+l, \cdot)-U_\xi(s+(n-1)\Delta t, \cdot) dl}_{p}\\
& \qquad \qquad \leq CT\Delta t+ \sum_{n=1}^N \int_{(n-1)\Delta t}^{n\Delta t} \norm{U_\xi (t+l, \cdot)-U_\xi (t+(n-1)\Delta t, \cdot)}_p dl\\
& \qquad \qquad \leq C\Delta t. 
\end{align*}
Since we can choose any $\Delta t>0$, it follows that 
\begin{equation*}
\norm{y_\xi(s,\cdot) -y_\xi(t,\cdot)- \int_t^s U_\xi(l, \cdot) dl}_p=0,
\end{equation*}
which means that $y_\xi(t,\cdot)-1$ satisfies the following integral equation in $L^2(\Real)\cap L^\infty(\Real)$
\begin{equation}\label{inteq:yxi}
y_\xi(s,\xi)=y_\xi(t,\xi)+ \int_t^s U_\xi(l, \xi) dl.
\end{equation}
 
 \subsection{The integral equation for $U_\xi$} Since $U(t,\xi)$ is a solution to \eqref{sys:intyUH2}, we can write for any $t<s$ and any $\xi<\bar\xi$, 
\begin{align}\label{diffn:U}
U(s,\bar\xi)-U(s, \xi)& = U(t, \bar\xi)-U(t, \xi)- \int_t^s Q(l, \bar\xi)-Q(l, \xi)dl.
\end{align}
Since we want to follow the same lines as for deriving \eqref{inteq:yxi}, we need to have a closer look at the integrand. Using integration by parts, $Q(t, \xi)$ can be written as 
\begin{align}\nonumber
Q(t,\xi)&=-\frac12 V(t,\xi)+ \frac14 \int_\Real e^{-\vert y(t,\xi)-y(t,\eta)\vert} (2UU_\xi+Vy_\xi)(t,\eta)d\eta\\ \label{def:hQ}
& = -\frac12 V(t,\xi)+\check P(t,\xi).
\end{align}

We have shown earlier that the functions $y$, $U$, and $H$ are Lipschitz continuous on $[0,T]\times \Real$, which implies that $y_\xi$, $U_\xi$, and $H_\xi$ are measurable and belong to $L^\infty([0,T]\times \Real)$. Thus, the set 
\begin{equation*}
S=\{(t,\xi)\in [0,T]\times \Real\mid y_\xi(t,\xi)=0\}
\end{equation*}
is Lebesgue measurable, which implies that $\mathbbm{1}_{S}(t,\xi)$, the indicator function of $S$, is measurable. Thus also the function
\begin{equation*}
V_\xi(t,\xi)=\mathbbm{1}_{S}H_\xi(t,\xi)
\end{equation*}
is measurable on $[0,T]\times \Real$. Moreover, according to Tonelli's theorem, see e.g. \cite[Proposition 5.2.1]{C}, the function 
\begin{equation*}
\xi \mapsto \int_{t}^s V_\xi(l, \xi)dl 
\end{equation*}
is measurable and for all $\xi \in \Real$
\begin{equation*}
 \int_{-\infty}^\xi \int_t^s V_\xi(l, \eta) dl d\eta= \int_t^s \int_{-\infty}^\xi V_\xi(l, \eta)d\eta dl= \int_t^s V(l, \xi) dl. 
\end{equation*} 

Consider now the function 
\begin{equation*}
g(\xi)= \int_{-\infty}^\xi \int_t^s V_\xi(l, \eta) dl d\eta= \int_t^s V(l, \xi)dl,
\end{equation*}
 which is increasing and Lipschitz continuous, and hence differentiable almost everywhere with derivative
 \begin{equation*}
 g_\xi(\xi)= \int_t^s V_\xi(l, \xi)dl \in L^2(\Real) \cap L^\infty(\Real),
 \end{equation*}
so that we can write 
\begin{align}\nonumber
\int_t^s Q(l, \bar\xi)-Q(l, \xi)dl &= -\frac12 \int_t^s V(l,\bar\xi)-V(l,\xi)dl+ \int_t^s \check P(l,\bar\xi)-\check P(l,\xi)dl\\ \label{diffn:U2}
& = -\frac12 (g(\bar\xi)-g(\xi))+ \int_t^s \check P(l,\bar\xi)-\check P(l,\xi)dl.
\end{align}

For the difference involving $\check P$, we use the following splitting
\begin{align}\nonumber
\check P(l, \bar\xi)-\check P(l, \xi)&= \frac14  (e^{y(t,\xi)-y(t,\bar\xi)}-1)\int_{-\infty}^{\xi}e^{y(l,\eta)-y(l,\xi)} (2UU_\xi+Vy_\xi)(l,\eta) d\eta\\ \nonumber
& \quad + \frac14(e^{y(t, \bar\xi)-y(t,\xi)}-1) \int_{\bar\xi}^\infty  e^{y(l, \xi)-y(l, \eta)} (2UU_\xi+Vy_\xi)(l, \eta) d\eta\\ \nonumber
& \quad + \frac14 \int_{\xi}^{\bar\xi} (e^{2y(l, \eta)-y(l, \xi)-y(l, \bar\xi)}-1)e^{y(l, \xi)-y(l, \eta)}(2UU_\xi+Vy_\xi)(l, \eta) d\eta\\ \nonumber
& \quad + \frac14 \int_{-\infty}^{\xi} (e^{y(l, \xi)-y(l, \bar\xi)}-e^{y(t, \xi)-y(t, \bar\xi)})e^{y(l, \eta) -y(l, \xi)} (2UU_\xi+Vy_\xi)(l, \eta) d\eta\\ \label{split:tiQ}
& \quad + \frac14 \int_{\bar\xi}^\infty (e^{y(l, \bar\xi)-y(l, \xi)}-e^{y(t, \bar\xi)-y(t, \xi)})e^{y(l, \xi)-y(l, \eta) } (2UU_\xi+Vy_\xi)(l, \eta) d\eta.
\end{align}
Furthermore, for any $t\in [0,T]$, Definition~\ref{def:Lagcoord} \eqref{cond:Lagcoord7} implies 
\begin{equation}\label{ne:est}
\vert UU_\xi(t,\xi)\vert \leq V_\xi(t,\xi)\leq H_\xi(t,\xi) \quad \text{ for almost every } \xi\in \Real
\end{equation}
and hence there exists $C>0$ such that 
\begin{align*}
4\vert \check P(t,\xi)\vert &\leq \int_\Real e^{-\vert y(t,\xi)-y(t,\eta)\vert} (2H_\xi+2Hy_\xi)(t,\eta) d\eta\leq C,
\end{align*}
so that, by \eqref{diffn:U}, \eqref{diffn:U2}, \eqref{split:tiQ}, \eqref{ident:relcl}, \eqref{Gron:k}, and \eqref{help:est},
\begin{align*}
\Big\vert (U&(s,\bar\xi)-U(s,\xi)) -(U(t,\bar\xi)-U(t, \xi))-\frac12(g(\bar\xi)-g(\xi))\\
& +\frac14  (e^{y(t,\xi)-y(t,\bar\xi)}-1)\int_t^s\int_{-\infty}^{\xi}e^{y(l,\eta)-y(l,\xi)} (2UU_\xi+Vy_\xi)(l,\eta) d\eta dl\\
& +\frac14(e^{y(t, \bar\xi)-y(t,\xi)}-1) \int_t^s \int_{\bar\xi}^\infty  e^{y(l, \xi)-y(l, \eta)} (2UU_\xi+Vy_\xi)(l, \eta) d\eta dl\Big\vert\\
 & \leq \frac14\int_t^s \int_{\xi}^{\bar\xi} \vert e^{2y(l, \eta)-y(l, \xi)-y(l, \bar\xi)}-1\vert e^{y(l, \xi)-y(l, \eta)}(2V_\xi+Vy_\xi)(l, \eta) d\eta dl\\
& \quad + \frac14 \int_t^s\int_{-\infty}^{\xi} \vert e^{y(l, \xi)-y(l, \bar\xi)}-e^{y(t, \xi)-y(t, \bar\xi)}\vert e^{y(l, \eta) -y(l, \xi)} (2H_\xi+Hy_\xi)(l, \eta) d\eta dl\\
& \quad + \frac14 \int_t^s\int_{\bar\xi}^\infty \vert e^{y(l, \bar\xi)-y(l, \xi)}-e^{y(t, \bar\xi)-y(t, \xi)}\vert e^{y(l, \xi)-y(l, \eta) } (2H_\xi+Hy_\xi)(l, \eta) d\eta dl\\
& \leq \frac14\int_t^s  \int_{\xi}^{\bar\xi} (y(l, \bar\xi)-y(l, \xi))e^{y(l, \bar\xi)-y(l, \eta)}(2V_\xi+Vy_\xi)(l, \eta) d\eta dl\\
&\quad + C \int_t^s  \vert (y(l, \bar\xi)-y(l, \xi))-(y(t, \bar\xi)-y(t, \xi))\vert e^{\max_{\bar l\in [t,s]}(k(\bar l, \bar\xi)-k(\bar l, \xi))}dl\\
& \leq C \int_t^s  (k(l, \bar\xi)-k(l, \xi) )e^{k(l, \bar\xi)-k(l, \xi)}dl\vert \bar\xi-\xi\vert\\
&\quad  + C(s-t) (e^{C(s-t)}-1) \| Z(t, \bar\xi)-Z(t, \xi)\|e^{C(\bar\xi-\xi)}\\
&\quad +C(s-t)e^{C(s-t)}\int_t^s \vert U(l, \xi)\vert + \max_{ \hat \xi\in [\xi,\bar\xi]} U^2(l,\hat \xi)+ \hat P(l, \xi)+\hat P(l, \bar\xi)dl e^{C(\bar\xi-\xi)}\vert \bar\xi-\xi\vert.
\end{align*}
Dividing now both sides by $\bar\xi-\xi$ and taking the limit as $\bar\xi\to\xi$, we end up with 
\begin{align*}
\vert U_\xi(s, \xi)&-U_\xi(t, \xi)-\frac12 \int_t^s V_\xi(l, \xi) dl\\
& \quad - \frac14\int_t^s \int_\Real \sign(\xi-\eta) e^{-\vert y(l, \xi)-y(l, \eta)\vert} (2UU_\xi+Vy_\xi)(l, \eta) d\eta dl y_\xi(t, \xi)\vert \\
& \leq C(s-t)^2 e^{C(s-t)} \|Z_\xi(t, \xi)\|\\
& \quad + C(s-t) e^{C(s-t)} \int_t^s \vert U(l, \xi)|+ U^2(l, \xi)+ 2\hat P(l, \xi)dl.
\end{align*}
Furthermore, integration by parts yields 
\begin{equation*}
\int_\Real \sign(\xi-\eta) e^{-\vert y(l, \xi)-y(l, \eta)\vert} (2UU_\xi+Vy_\xi)(l, \eta) d\eta= 2(U^2-2P)(l, \xi), 
\end{equation*}
and we can write
\begin{align*}
\vert U_\xi(s, \xi)&-U_\xi(t, \xi)-\frac12 \int_t^s V_\xi(l, \xi) dl- \frac12\int_t^s (U^2-2P)(l, \xi) dl y_\xi(t, \xi)\vert \\
& \leq C(s-t)^2 e^{C(s-t)} \|Z_\xi(t, \xi)\|\\
& \quad + C(s-t)e^{C(s-t)} \int_t^s \vert U(l, \xi)|+ U^2(l, \xi)+ 2\hat P(l, \xi)dl .
\end{align*}
Taking the $L^p$ norm for $p=2$ or $p=\infty$ on both sides, applying the Minkowski inequality for integrals and recalling Lemma~\ref{lem:finest}, \eqref{est:Pp}, and \eqref{est:Up}, we finally have 
\begin{align*}
 \norm{U_\xi(s, \cdot)-U_\xi(t, \cdot)-\frac12 \int_t^s V_\xi(l, \cdot) dl- \frac12\int_t^s(U^2-2P)(l, \cdot) dl y_\xi(t, \cdot)}_p \leq C(s-t)^2.
\end{align*}
Using once more Minkowski's inequality for integrals, we can also conclude that for any $t<s$,
\begin{equation*}
U_\xi(s,\cdot)-U_\xi(t, \cdot)-\frac12\int_t^s V_\xi(l, \cdot)dl- \frac12 \int_t^s  (U^2-2P)y_\xi(l, \cdot) dl \in L^2(\Real)\cap L^\infty(\Real).
\end{equation*}

Thus for any $\Delta t$ such that $N\Delta t= s-t$, we have using Lemma~\ref{lem:finest} once more
\begin{align*}
& \norm{U_\xi(s,\cdot)-U_\xi(t, \cdot)-\frac12\int_t^s V_\xi(l, \cdot)dl- \frac12 \int_t^s  (U^2-2P)y_\xi(l, \cdot) dl}_p\\
& \leq \sum_{n=1}^N \|U_\xi(t+n\Delta t,\cdot)-U_\xi(t+(n-1)\Delta t, \cdot)\\
& \qquad -\frac12\int_{(n-1)\Delta t}^{n\Delta t} V_\xi(t+l, \cdot)dl- \frac12 \int_{(n-1)\Delta t}^{n\Delta t}  (U^2-2P)y_\xi(t+l, \cdot) dl\|_p\\
& \leq \sum_{n=1}^N\|U_\xi(t+n\Delta t,\cdot)-U_\xi(t+(n-1)\Delta t, \cdot)\\
& \qquad -\frac12\int_{(n-1)\Delta t}^{n\Delta t} V_\xi(t+l, \cdot)dl - \frac12 \int_{(n-1)\Delta t}^{n\Delta t}  (U^2-2P)(t+l, \cdot) dly_\xi(t+ (n-1)\Delta t, \cdot)\|_p\\
& \quad + \sum_{n=1}^N \norm{\frac12 \int_{(n-1)\Delta t}^{n\Delta t}  (U^2-2P)(t+l, \cdot)(y_\xi(t+l, \cdot)-y_\xi(t+(n-1)\Delta t, \cdot)) dl}_p\\
& \leq CT\Delta t + C\sum_{n=1}^N \int_{(n-1)\Delta t}^{n\Delta t} \norm{y_\xi(t+l, \cdot)-y_\xi(t+(n-1)\Delta t, \cdot)}_p dl\\
& \leq C\Delta t.
\end{align*}
Since we can choose any $\Delta t>0$, it follows that 
\begin{equation*}
 \norm{U_\xi(s,\cdot)-U_\xi(t, \cdot)-\frac12 \int_t^s V_\xi(l, \cdot)dl- \frac12 \int_t^s  (U^2-2P)y_\xi(l, \cdot) dl}_p=0,
 \end{equation*}
 which means that $U_\xi(t, \cdot)$ satisfies the following integral equation in $L^2(\Real)\cap L^\infty(\Real)$
 \begin{equation*}
 U_\xi(s,\xi)=U_\xi(t, \xi)+\frac12 \int_t^s V_\xi(l, \xi)dl+ \frac12 \int_t^s  (U^2-2P)y_\xi(l, \xi) dl.
\end{equation*}

\subsection{The integral equation for $H_\xi$} Since $H(t,\xi)$ is a solution to \eqref{sys:intyUH3}, we can write for any $s<t$ and any $\xi<\bar\xi$,
\begin{align*}
H(s, \bar\xi)-H(s,\xi)&= H(t,\bar\xi)-H(t, \xi)+ \int_t^s (U^3-2PU)(l, \bar\xi)-(U^3-2PU)(l, \xi)dl\\
& = H(t, \bar\xi)-H(t, \xi)\\
& \quad + (U(t, \bar\xi)-U(t, \xi))\\
& \qquad \qquad \qquad \times \int_t^s (U^2(l, \bar\xi)+ U(l, \bar\xi)U(l, \xi)+U^2(l, \xi)- 2 P(l, \bar\xi) dl\\
& \quad -2\int_t^s (P(l, \bar\xi)-P(l, \xi)) U(l, \xi) dl\\
& \quad +\int_t^s (U^2(l, \bar\xi)+U(l, \bar\xi)U(l, \xi) + U^2(l, \xi)-2P(l, \bar\xi))\\
& \qquad \qquad \qquad \times ((U(l, \bar\xi)-U(l, \xi))-(U(t, \bar\xi)-U(t, \xi)))dl.
\end{align*}
Observing that $P$ and $\check P$, given by \eqref{def:P} and \eqref{def:hQ}, have the same structure, we can use a similar splitting to the one for $\check P$ in \eqref{split:tiQ} and obtain, using \eqref{est:uinf}, \eqref{est:pxpinf}, Definition~\ref{def:Lagcoord} \eqref{cond:Lagcoord7}, \eqref{Gron:k}, and \eqref{help:est},
\begin{align*}
 & \vert (H(s,\bar\xi)-H(s,\xi)) -(H(t, \bar\xi)- H(t, \xi))\\
& \qquad - (U(t, \bar\xi)-U(t, \xi))\int_t^s (U^2(l, \bar\xi)+ U(l, \bar\xi)U(l, \xi)+ U^2(l, \xi)-2P(l, \bar\xi)) dl \\
& \qquad +\frac12 ( e^{y(t, \xi)-y(t, \bar\xi)}-1)\int_t^s U(l, \xi)\int_{-\infty}^{\xi} e^{y(l, \eta) -y(l, \xi)}(U^2y_\xi+V_\xi)(l, \eta) d\eta dl\\
&\qquad +\frac12 (e^{y(t, \bar\xi)-y(t, \xi)}-1)\int_t^s U(l, \xi) \int_{\bar\xi}^\infty e^{y(l, \xi)- y(l, \eta)}(U^2y_\xi+V_\xi)(l, \eta) d\eta dl\vert \\
&\quad  \leq  \int _t^s 2(U^2(l, \bar\xi)+U^2(l, \xi)+ P(l, \bar\xi)) \vert (U(l, \bar\xi)-U(l, \xi))-(U(t, \bar\xi)- U(t, \xi))\vert dl \\
& \qquad +C \int_t^s   \int_{\xi}^{\bar\xi} (y(l, \bar\xi)-y(l, \xi)) e^{y(l, \bar\xi)-y(l, \xi)} (U^2y_\xi+V_\xi)(l, \eta) d\eta dl\\
& \qquad + C \int_t^s \vert (y(l, \bar\xi)-y(l, \xi))-(y(t, \bar\xi)- y(t, \xi))\vert e^{\max _{\bar l\in [s,t]}(k(\bar l, \bar\xi)-k(\bar l, \xi))} dl\\
& \quad \leq C\int_t^s \norm{(Z(l, \bar \xi)-Z(l, \xi))-(Z(t, \bar\xi)-Z(t, \xi))} dl \\
& \qquad + C \int_t^s (k(l, \bar\xi)-k(l, \xi))\int_{\xi}^{\bar\xi} 2V_\xi (l, \eta) d\eta dl\\
& \quad \leq C(e^{C(s-t)}-1) (s-t) \norm{Z(t, \bar\xi)-Z(t, \xi)}\\
& \qquad + Ce^{C(s-t)} (s-t)\int_t^s \vert U(l, \xi)\vert+ \max_{\hat \xi \in [\xi,\bar\xi]} U^2(l, \hat \xi)+ \hat P(l, \xi)+ \hat P(l, \bar\xi) dl \vert \bar\xi-\xi \vert \\
& \qquad + C \int_t^s (k(l, \bar\xi)-k(l, \xi))\int_{\xi}^{\bar\xi} 2V_\xi (l, \eta) d\eta dl.
\end{align*}
Dividing now both sides by $\bar\xi-\xi$ and taking the limit as $\bar\xi\to\xi$, we end up with 
\begin{align*}
\vert H_\xi(s, \xi)-H_\xi(t, \xi)& -\int_t^s 3U^2(l, \xi)-2P(l, \xi) dl U_\xi(t,\xi)+2\int_t^s QU(l, \xi) dl y_\xi(t, \xi) \vert \\
& \qquad \leq C(s-t)^2 \norm{Z_\xi(t, \xi)}\\
& \qquad \quad + Ce^{C(s-t)}(s-t) \int_t^s \vert U(l, \xi)\vert + U^2(l, \xi)+ 2\hat P(l, \xi)dl.
\end{align*}
Taking the $L^p$ norm for $p=2$ or $p=\infty$ on both sides, applying the Minkowski's inequality for integrals, recalling Lemma~\ref{lem:finest}, \eqref{est:Pp}, and \eqref{est:Up}, we finally have 
\begin{align*}
\Big\|H_\xi(s, \xi) -H_\xi(t, \xi) & - \int_t^s 3U^2(l, \xi)-2P(l, \xi) dlU_\xi(t,\xi)\\
& \qquad \qquad \qquad +2\int_t^s QU(l, \xi) dl y_\xi(t, \xi) \Big\|_p \leq C(s-t)^2.
\end{align*}
Using once more Minkowski's inequality for integrals, we can also conclude that for any $t<s$,
\begin{equation*}
H_\xi(s, \cdot)-H_\xi(t, \cdot)-\int_t^s (3U^2U_\xi-2PU_\xi-2QUy_\xi)(l, \cdot)dl \in L^2(\Real)\cap L^\infty(\Real).
\end{equation*}
Thus, for any $\Delta t$ such that $N\Delta t= t-s$, we have using Lemma~\ref{lem:finest}, \eqref{est:uinf}, and \eqref{est:pxpinf} once more,
\begin{align*}
& \norm{ H_\xi(s, \cdot)-H_\xi(t, \cdot)-\int_t^s (3U^2U_\xi-2PU_\xi-2QUy_\xi)(l, \cdot)dl}_p\\
& = \sum_{n=1}^N \|H_\xi(t+ n\Delta t, \cdot)- H_\xi(t+(n-1)\Delta t, \cdot) \\
& \qquad \qquad  - \int_{(n-1)\Delta t}^{n\Delta t} (3U^2-2P)U_\xi(t+l, \cdot) dl+ 2 \int_{(n-1)\Delta t}^{n\Delta t} QUy_\xi(t+l, \cdot)dl \|_p\\
& \leq \sum_{n=1}^N \|H_\xi(t+ n\Delta t, \cdot)-H_\xi(t+ (n-1)\Delta t, \cdot)\\
& \qquad \qquad - \int_{(n-1)\Delta t}^{n\Delta t} (3U^2-2P)(t+l, \cdot)dl U_\xi(t+ (n-1)\Delta t, \cdot)\\
& \qquad \qquad + 2\int_{(n-1)\Delta t}^{n\Delta t} QU(t+l, \cdot)dl y_\xi(t+(n-1)\Delta t, \cdot) \|_p\\
& \qquad \quad + \sum_{n=1}^N  \norm{\int_{(n-1)\Delta t}^{n\Delta t} (3U^2-2P)(t+l, \cdot)(U_\xi(t+l, \cdot)- U(t+(n-1)\Delta t, \cdot)) dl}_p\\
& \qquad \qquad + \sum_{n=1}^N 2 \norm{\int_{(n-1)\Delta t}^{n\Delta t} QU(t+l, \cdot) (y_\xi(t+l, \cdot)-y_\xi(t+(n-1)\Delta t, \cdot))dl}_p\\
& \leq CT\Delta t + C\sum_{n=1}^N \int_{(n-1)\Delta t}^{n\Delta t} \norm{y_\xi(t+l, \cdot)-y(t+ (n-1)\Delta t, \cdot)}_p dl\\
& \qquad \qquad + C\sum_{n=1}^N \int_{(n-1)\Delta t} ^{n\Delta t} \norm{U_\xi(t+l, \cdot)-U(t+ (n-1) \Delta t, \cdot)}_p dl\\
& \leq C\Delta t.
\end{align*}
Since we can choose any $\Delta t>0$, it follows that 
\begin{equation}
\norm{H_\xi(s, \cdot)-H_\xi(t, \cdot)- \int_t^s (3U^2U_\xi-2PU_\xi-2QUy_\xi)(l, \cdot)dl }_p =0, 
\end{equation}
which means that $H_\xi(t, \cdot)$ satisfies the following integral equation in $L^2(\Real)\cap L^\infty(\Real)$
\begin{equation*}
H_\xi(s, \xi)=H_\xi(t, \xi)+ \int_t^s (3U^2U_\xi-2PU_\xi-2QUy_\xi)(l, \xi ) dl.
\end{equation*}

\subsection{The final step in the proof of Theorem~\ref{main}} So far we have seen that $(y,U,V,H)$ are solutions to the system of integral equations corresponding to \eqref{sys:diffeq1}--\eqref{sys:diffeq3}, \eqref{def:Pb}, \eqref{def:Qb}, and \eqref{sys:diffeqxi1}--\eqref{sys:diffeqxi3} with 
\begin{equation*}
V(t,\xi)= \int_{-\infty}^\xi \mathbbm{1}_{\{ \zeta \in \Real\mid y_\xi(t,\zeta)\not=0\}}(\eta) H_\xi(t,\eta) d\eta.
\end{equation*}
Thus it remains to show that $V(t,\xi)$ satisfies \eqref{sys:diffeq4} on $[0,T]\times \Real$, i.e., 
\begin{equation}\label{equiv:V}
V(t,\xi)= \int_{-\infty}^\xi (1-\mathbbm{1}_{\{\zeta \mid \tau(\zeta)\leq t\}}(\eta))H_\xi(t,\eta) d\eta.
\end{equation}

Introduce
\begin{equation*}
\Omega(t)= \{\xi\in \Real\mid y_\xi(t,\xi)=0  \text{ or } y(t,\xi) \text{ is not differentiable}\},
\end{equation*}
 then \eqref{equiv:V} holds if and only if 
\begin{equation*}
\Omega(t)=\{\xi \mid \tau(\xi)\leq t \text{ or } y(t,\xi) \text{ is not differentiable}\} \quad \text{ for all } t\in [0,T].
\end{equation*}
Furthermore, by \eqref{def:tau}, it suffices to show that for any $s,t\in[0,T]$ with $t<s$
\begin{equation*}
\xi \in\Omega(s) \quad \text{ for almost every } \xi \in \Omega(t).
\end{equation*}

The function $\check P(t,\xi)$, given by \eqref{def:hQ}, can be written as $\check p(t,y(t,\xi))$, where 
\begin{equation*}
\check p(t,x)=\frac14  \int_{\Real } e^{-\vert x-y\vert } (2uu_x+ F)(t, y) dy,
\end{equation*}
with derivative 
\begin{equation*}
\check p_x(t,x)= -\frac14 \int_{\Real} \mathrm{sign}(x-y) e^{-\vert x-y\vert} (2uu_x+ F)(t, y) dy.
\end{equation*}
Moreover, 
\begin{align*}
\norm{\check p_x(t,\cdot)}_\infty\leq \nu(t, \Real) \leq \nu(0,\Real)= C \quad \text{ for all } t\in [0,T],
\end{align*} 
which implies that for any $\xi<\bar\xi$
\begin{equation*}
\vert \check P(t, \bar \xi)-\check P(t, \xi)\vert =\vert \check p(t, y(t, \bar \xi))-\check p(t, y(t, \xi))\vert \leq C \vert y(t, \bar \xi)-y(t,\xi)\vert.
\end{equation*}
Thus, by \eqref{def:hQ}
\begin{equation}\label{est:Qcrit}
\vert Q(t,\bar\xi)-Q(t,\xi)+ \frac12 (V(t,\bar\xi)-V(t,\xi))\vert \leq C\vert y(t,\bar\xi)-y(t,\xi)\vert.
\end{equation}

Next we establish that for any $s>t$, 
\begin{equation*}
U_\xi(s,\xi)\geq 0 \quad \text{ for almost every }\xi \in \Omega(t).
\end{equation*}
 Let $\xi<\bar\xi$ and recall the integral equation \eqref{sys:intyUH2}, which combined with \eqref{est:Qcrit} yields 
\begin{align}\nonumber
-(U(s,\bar\xi)-U(s,\xi))
& \leq -(U(t, \bar\xi)-U(t, \xi)) -\frac12 \int_t^s V(l, \bar\xi)-V(l, \xi)dl \\ \label{differ:U}
& \quad + C\int_t^s y(l, \bar\xi)-y(l, \xi)dl,
\end{align}
where the first integral on the right hand side is negative. 

Using once more \eqref{sys:intyUH1}--\eqref{sys:intyUH2} and \eqref{est:Qcrit}, we obtain for any $t\leq l\leq s$
\begin{align*}
y(l,\bar\xi)-y(l,\xi)
& = y(t, \bar\xi)-y(t,\xi)+ (l-t)(U(t, \bar\xi)-U(t, \xi))\\
& \quad -\int_t^l \int_t^m Q(n, \bar\xi)-Q(n, \xi) dn dm\\
& \leq y(t, \bar\xi)-y(t, \xi)+(l-t) (U(t, \bar\xi)-U(t, \xi))\\
& \quad +\frac12 \int_t^l \int_t^m V(n, \bar\xi)-V(n, \xi) dn dm\\
& \quad + C\int_t^l\int_t^m y(n, \bar\xi)-y(n, \xi)dn dm\\
& \leq y(t, \bar\xi)-y(t, \xi)+(s-t )\vert U(t, \bar\xi)-U(t, \xi)\vert\\
& \quad + \frac12 (s-t) \int_t^s V(n, \bar\xi)-V(n, \xi)dl \\
& \quad + C(s-t) \int_t^l y(n, \bar\xi)-y(n, \xi)dn,
\end{align*}
which implies, using Gronwall's inequality,
\begin{align*}
y(l, \bar\xi)-y(l, \xi)&\leq e^{C(s-t)(l-t)}\Big(y(t, \bar\xi)-y(t, \xi)+ (s-t)\vert U(t, \bar\xi)-U(t, \xi)\vert\\
&\qquad\qquad  \qquad  \quad +\frac12(s-t)\int_t^s V(n, \bar\xi)-V(n, \xi)dn\Big).
\end{align*}
Plugging now this very last inequality into \eqref{differ:U}, we end up with
\begin{align*}
-(U(s,\bar\xi)-U(s,\xi))&\leq -(U(t, \bar\xi)-U(t,\xi))+Ce^{C(s-t)^2}(s-t)^2 \vert U(t, \bar\xi)-U(t, \xi)\vert \\
& \qquad + Ce^{C(s-t)^2} (s-t) (y(t, \bar\xi)-y(t,\xi))\\
& \quad -\frac12(1-Ce^{C(s-t)^2}(s-t)^2) \int_t^s V(l, \bar\xi)-V(l, \xi)dl\\
& \leq  -(U(t, \bar \xi)-U(t, \xi))+Ce^{C(s-t)^2}(s-t)^2 \vert U(t, \bar\xi)-U(t, \xi)\vert \\
& \qquad + Ce^{C(s-t)^2} (s-t) (y(t, \bar\xi)-y(t,\xi)),
\end{align*}
if $Ce^{C(s-t)^2}(s-t)^2<1$.
Dividing now both sides by $\bar\xi-\xi$ and taking the limit as $\bar\xi\to\xi$, we finally have for almost every $\xi\in \Real$, 
\begin{equation*}
-U_\xi(s, \xi)\leq -U_\xi(t, \xi)+ Ce^{C(s-t)^2}(s-t) \big((s-t) \vert U_\xi(t,\xi)\vert + y_\xi(t, \xi)\big).
\end{equation*}
For $\xi\in \Omega(t)$, the right hand side of the above inequality equals $0$, and hence we have shown that 
\begin{equation}\label{cont:cond2}
U_\xi(s, \xi)\geq 0  \quad \text{ for almost every } \xi \in \Omega(t).
\end{equation} 

Furthermore, by \eqref{sys:intyUH1}, Definiton \ref{def:Lagcoord} \eqref{cond:Lagcoord7}, \eqref{ident:relcl}, and \eqref{Gron:k}, there exists a constant $M>0$ such that 
\begin{align*}
y(s,\bar\xi)-y(s,\xi)&= y(t, \bar\xi)-y(t,\xi)+ \int_t^s U(l, \bar\xi)-U(l, \xi)dl\\
& \leq y(t, \bar\xi)-y(t,\xi)+ \int_t^s k(l, \bar\xi)-k(l, \xi)dl\\
& \leq y(t, \bar\xi)-y(t,\xi)+ \int_t^s e^{Ml} dl (\bar\xi-\xi)\\
& = y(t, \bar\xi)-y(t, \xi)+ \frac1{M} (e^{Ms}-e^{Mt}) (\bar\xi-\xi)\\
& \leq y(t, \bar\xi)-y(t, \xi)+ (s-t)e^{MT} (\bar\xi-\xi).
\end{align*}
Dividing now both sides by $\bar\xi-\xi$ and taking the limit as $\bar\xi\to\xi$, we finally have for almost every $\xi\in \Real$,
\begin{equation*}
y_\xi(s, \xi)\leq y_\xi(t,\xi)+ (s-t)e^{MT}.
\end{equation*}
Choosing now $\xi\in \Omega(t)$ in the above inequality, we end up with 
\begin{equation}\label{cont:cond}
y_\xi(s,\xi)\leq (s-t) e^{MT} \quad \text{ for almost every } \xi\in \Omega(t).
\end{equation}

Finally, we relate $y_\xi(t,\xi)$ to the pair $(u,\nu)(t)$. Therefore pick a relabeling function $g$, then the mapping $L_g: \D\to \F$ given by 
\begin{align*}
 \hat y(\xi)&= \sup\{x\mid x+ \nu((-\infty,x))<g(\xi)\},\\
 \hat H(\xi)& =g(\xi)-\hat y(\xi),\\
 \hat U(\xi)&= u(\hat  y(\xi)),\\
 \hat V(\xi)&= \int_{-\infty}^{\hat y(\xi)} (u^2+u_x^2)(z)dz = \int_{-\infty}^{\xi} \hat H_\xi(\eta) \mathbbm{1}_{\{\zeta\mid \hat y_\xi(\zeta)\not=0\}}(\eta) d\eta
\end{align*}
 associates to each element $(u,\nu)$ a quadraple $(\hat y, \hat U,\hat V,\hat H)$. Furthermore, a closer look at Definition~\ref{def:EultoLag} reveals that 
\begin{equation*}
L_g((u,\nu))=L((u,\nu))\bullet g, 
\end{equation*}
i.e., $L$ and $L_g$ map to the same equivalence class, but different representatives, in Lagrangian coordinates. In addition, most of the properties of $L$ carry over to $L_g$. In particular, choosing $g(\xi)=k(t,\xi)$, we have $L_{k(t,\cdot)}((u(t), \nu(t))=(y,U,V,H)(t)$. Furthermore, one can establish as in the proof of  \cite[Theorem 3.8]{HR2} that for almost every $\xi$ either $y_\xi(t,\xi)=U_\xi(t,\xi)=0$ and $y(t,\xi)\in \mathrm{supp}(\nu_{\mathrm{sing}}(t))$ or $ U_\xi(t,\xi)= u_x(t,y(t,\xi)) y_\xi(t,\xi)$ and 
\begin{equation}\label{rel:yxi}
(1+u_x^2(t, y(t,\xi))) y_\xi(t,\xi)=k_\xi(t,\xi).
\end{equation}
Moreover, one has for almost every $\xi$ that $u_x(t,y(t,\xi))>0$ implies $U_\xi(t,\xi)\geq 0$, while $u_x(t,y(t,\xi))<0$ yields $U_\xi(t,\xi)\leq 0$. Recalling \eqref{rel:yxi}, Lemma~\ref{lem:k}, and \eqref{oneside:Lip}, it follows that if $U_\xi(t,\xi)\geq 0$, then either
\begin{equation}\label{cont:cond3}
y_\xi(t, \xi)=0 \quad \text{ or } \quad y_\xi(t,\xi)\geq \frac{e^{-Mt}}{1+D^2}>0. 
\end{equation}
Comparing now \eqref{cont:cond} and \eqref{cont:cond2} with \eqref{cont:cond3} yields that all of them can only be satisfied if for all $t\leq  s$
\begin{equation*}
y_\xi(t,\xi)=0 \quad \text{ implies }\quad  y_\xi(s,\xi)=0 \quad \text{ for almost every } \xi \in \Real,
\end{equation*}
which is equivalent to $\xi \in \Omega(s)$ for almost every $\xi\in\Omega(t)$. 

\vspace{0.2cm}
To complete the proof of Theorem~\ref{main} one observation is important. When deriving a lower bound on $U_\xi(t,\xi)$ in this section, we encountered for the first time a condition on the choice of $T$, namely $CT^2 e^{CT^2}<1$, where $C=\nu(0, \Real)$ and hence independent of time. Thus we have so far only shown that on a small time interval $[0,T]$ the solution $(u,\nu)$ coincides with the dissipative solution constructed in \cite{GHR2} and hence is unique. But the above argument can be carried out on any interval $[T_1, T_2]$, whose length $T_2-T_1$ satisfies $C(T_2-T_1)^2 e^{C(T_2-T_1)^2}<1$, so that   considering the chain of intervals $[0, T]$, $[\frac12 T, \frac32 T]$, $[T, 2T]$, $[\frac32T, \frac52T]$, $\dots$, finishes the proof of Theorem~\ref{main}.

\section{Uniqueness of weak dissipative solutions $u$} \label{subsec:equivcl}

The set $\D$ of Eulerian coordinates equipped with the equivalence relation given by Definition~\ref{def:equivEuler} allows to identify each $u\in H^1_u(\Real)$ with an equivalence class in $\D$. If the solution operator from \cite{GHR2}, which associates to each pair $(u_0,\nu_0)\in \D$ the corresponding unique weak dissipative solution $(u,\nu)$, see Theorem~\ref{main}, respects this equivalence relation, the following result holds.

\begin{theorem}\label{main2}
For any initial data $u_0\in H^1_u(\Real)$ the Camassa--Holm equation has a unique global weak dissipative solution $u$ in the sense of Definition~\ref{def:dissol}.
\end{theorem} 

It therefore remains to prove the following lemma.

\begin{lemma}\label{lem:main2}
Given two weak dissipative solutions $(u_A,\nu_A)$ and $(u_B, \nu_B)$ with initial data $(u_{0,A}, \nu_{0,A})$ and $(u_{0,B}, \nu_{0,B})$ in the sense of Definition~\ref{def:dissol2}. If 
\begin{equation}\label{ini:equivcla}
u_{0,A}=u_{0,B},
\end{equation} 
then 
\begin{equation*}
u_A(t)=u_B(t) \quad \text{ for all }t\geq 0.
\end{equation*}
\end{lemma}

\begin{proof}
 It suffices to prove the claim for $d\nu_{0,A}=(u_{0,A}^2+u_{0,A,x}^2) dx$, which we assume from now on.

Let $L((u_{0,i}, \nu_{0,i}))=(y_{0,i}, U_{0,i}, V_{0,i}, H_{0,i})$ for $i=A$, $B$. We claim there exists an increasing and Lipschitz continuous function $g$ such that 
\begin{equation}\label{prop:quasirel}
(y_{0,A}\circ g, U_{0,A}\circ g, V_{0,A}\circ g)=(y_{0,B}, U_{0,B}, V_{0,B}).
\end{equation}

Since $V_{0,A}=H_{0,A}$,
\begin{equation*}
y_{0,A}(\xi)+V_{0,A}(\xi)=\xi \quad \text{ for all } \xi \in \Real.
\end{equation*}
For $V_{0,B}(\xi)$, on the other hand, we have that 
\begin{equation}\label{rel:HBVB}
V_{0,B,\xi}(\xi)=\mathbbm{1}_{\{\zeta\mid y_{0,B,\xi}(\zeta)\not =0\}}(\xi) H_{0,B,\xi}(\xi) \quad \text{ for all } \xi \in \Real,
\end{equation}
which implies that 
\begin{align*}
y_{0,B}(\xi)+V_{0,B}(\xi)&= y_{0,B}(\xi)+H_{0,B}(\xi)+ V_{0,B}(\xi)-H_{0,B}(\xi)\\
& = \xi -\int_{-\infty} ^\xi (1-\mathbbm{1}_{\{ \zeta \mid y_{0,B, \xi} (\zeta)\not =0\}}(\eta))H_{0,B,\xi}(\eta)d\eta,
\end{align*}
where the function on the right hand side is increasing and Lipschitz continuous with Lipschitz constant at most one. Introduce 
\begin{equation}\label{def:g}
g(\xi)= \xi- \int_{-\infty} ^\xi (1-\mathbbm{1}_{\{ \zeta \mid y_{0,B, \xi} (\zeta)\not =0\}}(\eta))H_{0,B,\xi}(\eta)d\eta,
\end{equation}
then 
\begin{equation}\label{prop:quasirel2}
y_{0,B}(\xi)+V_{0,B}(\xi)=g(\xi)=y_{0,A}(g(\xi))+V_{0,A}(g(\xi)) \quad \text{ for all } \xi \in \Real.
\end{equation}

Having identified a candidate for the function $g$ we are looking for, it remains to show that \eqref{prop:quasirel} holds. If $y_{0,A}\circ g \not = y_{0,B}$, then there exists a $\bar \xi\in\Real$ such that $y_{0,A}(g(\bar \xi))\not= y_{0,B}(\bar \xi)$ and without loss of generality we assume 
\begin{equation}\label{contr}
y_{0,A}(g(\bar \xi))<y_{0,B}(\bar \xi).
\end{equation}

Definition~\ref{def:EultoLag} implies that 
\begin{equation*}
V_{0,i}(\xi)= \int_{-\infty}^{y_{0,i}(\xi)} (u_{0,i}^2+u_{0,i,x}^2) dx \quad \text{ for all } \xi \in \Real \text{ and } i=A, B,
\end{equation*}
which combined with \eqref{ini:equivcla} and \eqref{prop:quasirel2} yields 
\begin{align*}
\int_{-\infty}^{y_{0,B}(\bar \xi)} (u_{0,A}^2+u_{0,A,x}^2) dx & =\int_{-\infty}^{y_{0,B}(\bar \xi)}(u_{0,B}^2+u_{0,B,x}^2) dx= V_{0,B} (\bar \xi)\\
& < V_{0,A} (g(\bar \xi))\leq \int_{-\infty}^{y_{0,A}(g(\bar \xi))} (u_{0,A}^2+u_{0,A,x}^2) dx.
\end{align*}
Since the above inequality can only hold if $y_{0,B}(\bar\xi)\leq y_{0,A}(g(\bar\xi))$, we end up with a contradiction to \eqref{contr}. Thus $y_{0,A}\circ g=y_{0,B}$, $V_{0,A} \circ g= V_{0,B}$, and, by Definition~\ref{def:EultoLag},
\begin{equation*}
 U_{0,A} \circ g= u_{0,A}\circ y_{0,A} \circ g= u_{0,A} \circ y_{0,B}=u_{0,B} \circ y_{0,B}=U_{0,B},
 \end{equation*}
 which finishes the proof of \eqref{prop:quasirel}.

Next, we show
\begin{equation}\label{prop:relok}
(y_A, U_A, V_A)(t, g(\xi))= (y_B, U_B, V_B)(t, \xi) \quad \text{ for all } \xi \in \Real \text{ and } t\geq 0.
\end{equation}

The function $g$, given by \eqref{def:g}, is increasing and Lipschitz continuous, and hence differentiable almost everywhere with derivative 
\begin{equation}\label{expl:g}
g_\xi(\xi)=1+(V_{0,B,\xi}-H_{0,B,\xi})(\xi)
 = \begin{cases} 0,& \text{ if } \xi\in \hat S,\\
 1, & \text{ otherwise}
 \end{cases}
\end{equation}
where 
\begin{equation}\label{def:hatS}
\hat S=\{\xi \in \Real \mid y_{0,B, \xi}(\xi)=0\}. 
\end{equation}
Furthermore, for every $t\geq 0$, the functions $(y_A(t,\cdot), U_A(t,\cdot), H_A(t,\cdot))$ are Lipschitz continuous, which implies that also $V_A(t,\cdot)$ is Lipschitz continuous, since 
 \begin{equation*}
 \vert V_A(t,\xi_1)-V_A(t,\xi_2)\vert \leq \vert H_A(t,\xi_1)-H_A(t,\xi_2)\vert \quad \text{ for all } \xi_1, \xi_2\in \Real,
 \end{equation*}
 and hence 
 \begin{equation*}
 (y_{A}(t,g(\xi)))_\xi = (U_{A}(t,g(\xi)))_\xi= (V_{A}(t, g(\xi)))_\xi=0 \quad \text{for almost every } \xi \in \hat S.
 \end{equation*}
Finally, by \eqref{def:hatS}, \eqref{sys:diffeqxi}, and Definition~\ref{def:Lagcoord} \eqref{cond:Lagcoord7}, we have
\begin{align*}
(y_A(t,g(\xi)))_\xi&=0=y_{B,\xi}(t,\xi),\\
(U_A(t,g(\xi)))_\xi&=0=U_{B,\xi}(t,\xi),\\
(V_A(t,g(\xi)))_\xi&=0=V_{B,\xi}(t,\xi),
\end{align*}
for almost every $\xi\in \hat S$. 

Next, introduce the set
\begin{equation*}
Z=\hat S\cup \{\xi \in\Real \mid y_{0,B} \text{ or } g \text{ are not differentiable}\},
\end{equation*}
which satisfies $\meas(Z)=\meas(\hat S)$. Then the definition of $\hat S$ and the fact that $g$ is Lipschitz continuous imply that $\meas(g(Z))=0$. Furthermore, $y_{0,A}$ is differentiable almost everywhere on $g(Z)^c$ and hence 
\begin{equation*}
(y_{0,A}(g(\xi)))_\xi= y_{0,A,\xi}(g(\xi))g_\xi(\xi)=y_{0,B,\xi}(\xi) \quad \text{ for almost every } \xi \in Z^c.
\end{equation*}
To finish the proof of \eqref{prop:relok} a generalized relabeling argument as for showing \cite[Proposition 5.4]{HR} can be used, which we do not repeat here.

Finally, we can apply the mapping $M$ to go back to Eulerian coordinates as follows. Let $(t,x)\in \Real^+\times \Real$, then there exists $\xi \in \Real$ such that 
\begin{equation*}
y_A(t,g(\xi))=x=y_B(t,\xi),
\end{equation*}
and hence 
\begin{equation*}
u_A(t,x)=U_A(t, g(\xi))=U_B(t,\xi)=u_B(t,x).
\end{equation*}
\end{proof}

\appendix


\section{A peakon-antipeakon example}\label{pap}

Over the last 20 years the so-called peakon-antipeakon solution has attracted a lot of attention, since wave breaking occurs and the weak solution cannot be uniquely continued thereafter, see, e.g., \cite{W, GHR2, GH2}. Moreover, this solution can be computed explicitly and hence is a rich source of inspiration when developing analytical methods as well as numerical algorithms. 

Given $p(0)>0$ and $q(0)<0$, let 
\begin{equation}\label{ini:pap}
u(0,x)= p(0)(e^{-\vert x-q(0)\vert }- e^{-\vert x+ q(0)\vert}),
\end{equation}
and set
\begin{equation*}
D^2= p^2(0)(1-e^{2q(0)}).
\end{equation*}
Denoting by $t^*>0$ the time when wave breaking occurs, which is given by
\begin{equation*}
t^*=\frac{1}{2D} \ln\left(\frac{p(0)+D}{p(0)-D}\right),
\end{equation*}
the dissipative peakon-antipeakon solution with initial data
\eqref{ini:pap} reads
\begin{equation*}
u(t,x)= \begin{cases} p(t)(e^{-\vert x-q(t)\vert} - e^{-\vert x+ q(t)\vert}), & \quad (t,x) \in [0, t^*)\times \Real,\\
0, & \quad \text{otherwise},
\end{cases}
\end{equation*}
where $p(t)>0$ and $q(t)<0$ are given by 
\begin{equation*}
p(t)= D\frac{1+ e^{2D(t-t^*)}}{1-e^{2D(t-t^*)}} \quad \text{ and } \quad q(t)= \ln\left(\frac{2e^{D(t-t^*)}}{1+ e^{2D(t-t^*)}}\right).
\end{equation*}
In particular, 
\begin{equation*}
\lim_{t\to t^*} u(t,x)=0 \quad \text{ for all } x\in \Real
\end{equation*}
and $u(t,x)$ is continuous on $\Real^+\times \Real$.

\begin{figure}
\includegraphics[width=8cm]{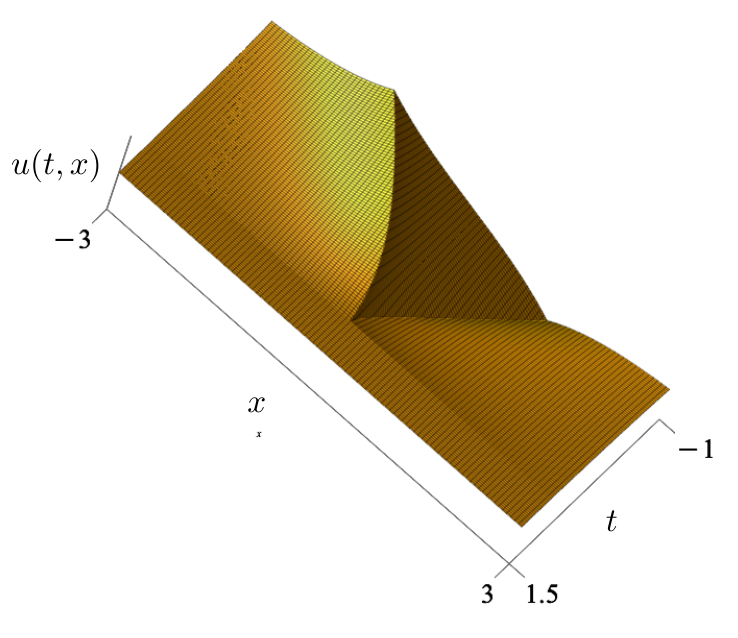}

\caption{Plot of the function $u(t,x)$ for $t\in [-1,1.5]$ with $t^*=1$ and $D=1$. Note that the time axis has a different orientation than usual.}
\end{figure}

The function 
\begin{equation*}
F(t,x)=\int_{-\infty}^x (u^2+u_x^2)(t,y)dy ,
\end{equation*}
which is given by 
\begin{equation*}
F(t,x)= \begin{cases} 
\begin{cases} 
D^2(1-e^{2q(t)})e^{2(x-q(t))}, & x<q(t),\\
2D^2 + 2p^2(t)e^{2q(t)}\sinh(2x), & q(t)<x<-q(t),\\
4D^2-D^2(1-e^{2q(t)})e^{-2(x+q(t))}, & -q(t)<x,
\end{cases} & t\in [0, t^*),\\
0, & t\in [t^*, \infty).
\end{cases}
\end{equation*}
is not continuous, since 
\begin{equation*}
\lim_{t\uparrow t^*}F(t,x)= 
\begin{cases} 
0, & x<0,\\ 4D^2 , & 0<x.
\end{cases}
\end{equation*}
A closer look reveals that $F(t,x)$ is continuous on $(\Real^+\times \Real )\backslash H$ and $F(t,x)$ has a jump of height $4D^2$ when crossing the half line $H$, which is given by $H=\{(t^*, x)\mid x\in [0, \infty)\}$. In addition, observe that for each $t\in \Real^+$, the function $F(t, \cdot)$ is absolutely continuous. 

\begin{figure}
\includegraphics[width=8cm]{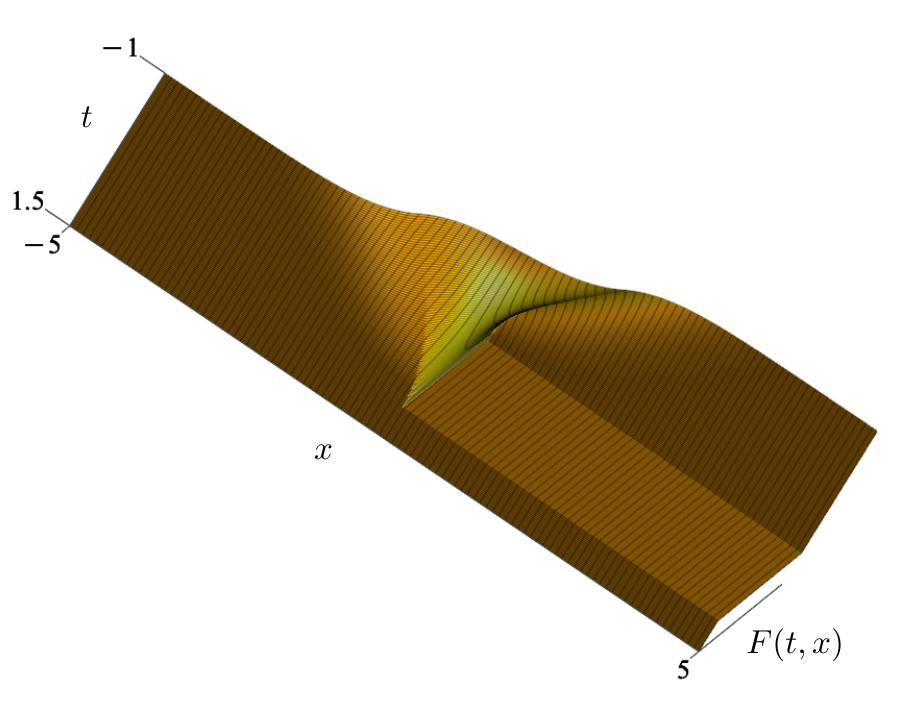}
\caption{Plot of the function $F(t,x)$ for $t\in [-1,1.5]$ with $t^*=1$ and $D=1$. Note that the time axis has a different orientation than usual.}
\end{figure}

Also the functions $p(t,x)$ and $p_x(t,x)$ can be computed explicitly. 
For any $t\in [0, t^*)$, the function $p(t,x)$  is given by 
\begin{equation*}
p(t,x)=
\begin{cases}
\frac14 \frac{D^4}{p^2(t)} (e^{x+q(t)}-2e^{2(x-q(t))} + 3e^{x-q(t)}) &\\
\qquad - \frac12 (p^2(t)-D^2)e^x(\sinh(q(t))+\sinh(3q(t))), & x<q(t),\\
\frac14 \frac{D^4}{p^2(t)}  (e^{q(t)-x}+ e^{x+q(t)}) & \\
\qquad + \frac12 (p^2(t)-D^2)\big(-2\cosh(2x)-2& \\ 
 \quad\qquad\qquad   + \cosh(x)(3e^{-q(t)}+2e^{q(t)}-e^{3q(t)})\big) , & q(t)<x<-q(t),\\
\frac14 \frac{D^4}{p^2(t)} (e^{-x+q(t)}-2e^{-2(x+q(t))}+ 3e^{-(x+q(t))}), & \\
\qquad - \frac12 (p^2(t)-D^2)e^{-x}(\sinh(q(t))+\sinh(3q(t))), & -q(t)<x,
\end{cases} 
\end{equation*}
while 
\begin{equation*}
p(t,x)=0 \quad \text{ for all } (t,x) \in [t^*,\infty)\times \Real.
\end{equation*}
Furthermore,
\begin{equation*}
\lim_{t\uparrow t^*} p(t,x)= D^2 e^{-\vert x\vert},
\end{equation*}
which implies that $p(t,x)$ is continuous on $(\Real^+\times \Real) \backslash L$ and has a jump when crossing the line $L= \{(t^*,x)\mid x\in \Real\}$. Observe that while $F(t,x)$ only has a jump along the half line $H$ starting at the point $(t^*, 0)$, $p(t,x)$ has a jump along the line $L$. Furthermore, for any $t\in \Real^+$, the function $p(t,\cdot)$ is continuous. Moreover, it can be shown that for each Lipschitz continuous curve $\sigma(t): [T_1, T_2] \to \Real$ with $0\leq T_1<T_2\leq T$, the function $g(t)=p(t, \sigma(t))$ has at most one jump at $t=t^*$ and is a function of bounded variation. 

\begin{figure}
\includegraphics[width=8cm]{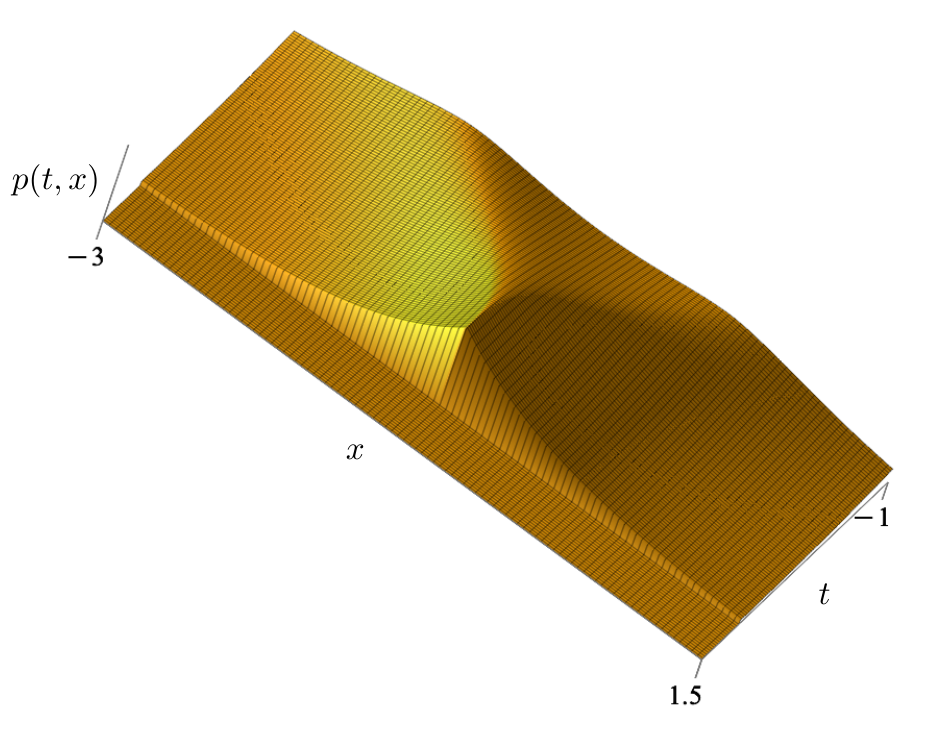}
\caption{Plot of the function $p(t,x)$ for $t\in [-1,1.5]$ with $t^*=1$ and $D=1$. Note that the time axis has a different orientation than usual.}
\end{figure}

For any $t\in [0, t^*)$, the function $p_x(t,x)$ is given by  
\begin{equation*}
p_x(t,x)=
\begin{cases}
\frac14 \frac{D^4}{p^2(t)} (e^{x+q(t)}-4e^{2(x-q(t))} + 3e^{x-q(t)}) &\\
\quad - \frac12 (p^2(t)-D^2)e^x(\sinh(q(t))+\sinh(3q(t))), & x<q(t),\\
\frac14 \frac{D^4}{p^2(t)}  (-e^{q(t)-x}+ e^{x+q(t)}) & \\
\quad + \frac12 (p^2(t)-D^2)\big(-4\sinh(2x)& \\ 
 \quad\qquad  + \sinh(x)(3e^{-q(t)}+2e^{q(t)}-e^{3q(t)})\big) , & q(t)<x<-q(t),\\
\frac14 \frac{D^4}{p^2(t)} (-e^{-x+q(t)}+4e^{-2(x+q(t))}-3e^{-(x+q(t))}), & \\
\quad + \frac12 (p^2(t)-D^2)e^{-x}(\sinh(q(t))+\sinh(3q(t))), & -q(t)<x,
\end{cases} 
\end{equation*}
while 
\begin{equation*}
p_x(t,x)=0 \quad \text{ for all } (t,x) \in [t^*,\infty)\times \Real.
\end{equation*}
Furthermore, 
\begin{equation*}
\lim_{t\uparrow t^*} p_x(t,x)= \mathrm{sign}(x) D^2e^{-\vert x\vert},
\end{equation*}
which implies that $p_x(t,x)$ is continuous on $(\Real^+\times \Real) \backslash L$ and has a jump when crossing the line $L= \{(t^*,x)\mid x\in \Real\}$. Again, observe that while $F(t,x)$ only has a jump along the half line $H$ starting at the point $(t^*, 0)$, $p_x(t,x)$ has a jump along the line $L$. Furthermore, for any $t\in \Real^+$, the function $p_x(t,\cdot)$ is continuous.

%
%
%

\end{document}